\documentclass[abstracton, paper=a4, fontsize=11pt, DIV=12]{scrartcl}

\pdfoutput=1

\usepackage[T1]{fontenc}
\usepackage{verbatim}
\usepackage{amsmath,amsthm,amssymb,amsfonts}
\usepackage{amsthm}
\usepackage{mathrsfs}
\usepackage[utf8]{inputenc}
\usepackage{enumitem}
\usepackage{caption}
\usepackage{mathtools}
\usepackage{esint}
\usepackage{textcomp}
\usepackage[backend=bibtex,style=alphabetic,isbn=false, doi=false, eprint= false, url= false,]{biblatex}
\usepackage{csquotes}
\usepackage{nicefrac}

    \makeatletter
    \renewenvironment{proof}[1][\proofname]{%
      \par\pushQED{\qed}\normalfont%
      \topsep6\p@\@plus6\p@\relax
      \trivlist\item[\hskip\labelsep\bfseries#1\@addpunct{.}]%
      \ignorespaces
    }{%
      \popQED\endtrivlist\@endpefalse
    }
    \makeatother

\newcommand{\dist}{\operatorname{dist}}
\newcommand{\dv}{\operatorname{div}}
\newcommand{\op}{\operatorname}
\newcommand{\IR}{\mathbb{R}}
\newcommand{\IN}{\mathbb{N}}

\newcommand{\step}[1]{\par\medskip\par\noindent\textit{#1}} 
\newcommand{\dd}{\, \mathrm{d}}

\newcommand{\supp}{\operatorname{supp}}

\newcommand{\eps}{\varepsilon}

\newcommand{\Hdot}{\dot{H}^1}

\newtheorem{theorem}{Theorem}[section]
\newtheorem{lemma}[theorem]{Lemma}

\newtheorem{proposition}[theorem]{Proposition}

\theoremstyle{definition}
\newtheorem{definition}[theorem]{Definition}
\newtheorem{assumption}[theorem]{Assumption}
\newtheorem*{notation}{Notation}

\newtheorem{remark}[theorem]{Remark}

\mathtoolsset{showonlyrefs}

\bibliography{bibliography.bib}

\begin{document}

\begin{center}
{\Large Sedimentation of Inertialess Particles in Stokes Flows}

\bigskip

Richard M. H\"{o}fer\footnote{University Bonn, Endenicher Allee 60, 53115 Bonn, Germany}

\vspace{2mm}

\today

\end{center}

\begin{abstract}
We investigate the sedimentation of a cloud of rigid, spherical particles of identical radii under gravity
in a Stokes fluid. Both inertia and rotation of particles are neglected.
We consider the homogenization limit of many small particles in the case of a dilute system
in which interactions between particles are still important.
In the relevant time scale, we rigorously prove convergence of the dynamics
to the solution of a macroscopic equation.
This macroscopic equation resembles the Stokes equations for a fluid of variable density subject to gravitation.
\end{abstract}

\section{Introduction}

We consider a cloud of  $N$ spherical particles with identical radii $R$ sedimenting in a fluid.
We denote the positions of the centers of the particles by $(\bar{X}_i)_{1\leq i\leq N}$  and their
velocities by $(\bar{V}_i)_{1\leq i\leq N}$.

The fluid surrounding the particles is assumed to satisfy Stokes equations with no-slip boundary conditions at the particles, neglecting particle rotations, i.e.,
\begin{equation}
\label{eq:StokesBdry}
\begin{aligned}
	- \mu \Delta  \bar{v} + \nabla q &= \rho_f g \quad \text{in} ~ \IR^3 \backslash \bigcup_{i=1}^N \overline{B_i}, \\
	\dv \bar{v} &= 0 \quad \text{in} ~ \IR^3 \backslash \bigcup_{i=1}^N \overline{B_i},\\
	\bar{v}(x) &\to 0  \quad \text{as} ~ |x| \to \infty, \\
	\bar{v} &= \bar{V}_i \quad \text{on} ~ \partial {B_i} \quad \text{for all} ~ 1 \leq i \leq N \\
	\bar{V}_i &= \dot{\bar{X}}_i \quad \text{for all} ~ 1 \leq i \leq N.
\end{aligned}
\end{equation}
Here, $\bar{v}$ denotes the fluid velocity, $\rho_f$ its density, and $\mu$ its viscosity, $q$ is the pressure, $g$ the gravitational acceleration, and $B_i := B_{R_i}(\bar{X}_i)$.

Assuming inertialess particles of identical density $\rho_p$ means
\begin{equation}
	\label{eq:Inertialess}
	\frac{4 \pi}{3} R^3 \rho_p  g = - \int_{\partial B_i} \sigma n \dd \mathcal{H}^2 \quad \text{for all} ~ 1 \leq i \leq N,
\end{equation}
where $ \sigma = \mu (\nabla \bar{v} + (\nabla \bar{v})^T) - q I$ is the stress, and $n$ denotes the 
unit outer normal.

Thus, we consider the following problem. Given initial particle positions $\bar{X}_i^0$ and radii $R_i$,
 we have to determine $\bar{X}_i(t)$, $\bar{V}_i(t)$,
$\bar{v}(x,t)$, and $p(x,t)$ such that \eqref{eq:StokesBdry} and \eqref{eq:Inertialess} hold 
for $t > 0$ and 
\begin{equation}
	\label{eq:initialData}
	\bar{X}_i(0) = \bar{X}_i^0  \quad \text{for all} ~ 1 \leq i \leq N.
\end{equation}

To understand the macroscopic behavior of a system with many particles in a viscous fluid is an important and challenging problem in numerous applications,
and is far from being understood [Gua06].
In the static case, it has been shown in \cite{Al90a}, \cite{Al90b}, and \cite{DGR08} that in the homogenization limit, the fluid satisfies
Brinkman equations or Darcy's law (depending on the density and radii of the particles).

In the dynamic case, so called Vlasov-Stokes (or Vlasov-Navier-Stokes)
equations are proposed if the inertia of particles (and the inertia of the fluid) is taken into account
(see \cite{Ham98} and \cite{GJV04b} and the references therein).
While there are results on existence and certain limits of these macroscopic Vlasov type equations, 
a rigorous mathematical derivation of these equations from the microscopic level involving particles
is still lacking.

If the volume fraction of the particles is sufficiently small, inertia of both the fluid and the particles
can be expected to be negligible (see below). 
In this case the identification of the regime that is so dilute that particle interactions are negligible
has been provided in \cite{JO04}.
In the present work, we will rigorously derive a macroscopic limit in the case of inertialess particles
of small volume fraction, but in systems that are not so dilute as in \cite{JO04}.
This macroscopic equations can be viewed as the inertialess limit of the Vlasov type equations mentioned above.

\subsection{Interactions between particles}
\label{sec:interactions}

The particles in the fluid interact with each other by affecting the 
fluid velocity which in turn determines the velocities of the particles.
More precisely, as we will see, given particle positions $\bar{X}_i(t) $ at some fixed time $t$,
the fluid velocity $v(t,\cdot)$ is uniquely determined by satisfying Stokes equations, equation \eqref{eq:Inertialess} and the constraint
\[
	v(t,\cdot) = \mathrm{const} \quad \text{on} ~ \partial {B_i} \quad \text{for all} ~ 1 \leq i \leq N,
\]
which is part of the fourth equation in \eqref{eq:StokesBdry}.
Due to the long range structure of the Stokes equations, interaction between particles in this way
are also long range.

The velocity of each particle can be interpreted as consisting of the sum of two contributions.  The first one is the self-interactive
part which is due to the direct influence of the gravitational force on the particle itself.
It corresponds to the velocity of the particle in the absence of other particles.
The second part is the collective effect due to all the other particles. Indeed, their motion results in a macroscopic fluid velocity which again affects each particle.

As in similar problems with interacting particles, there is 
an intrinsic length scale,
that determines how strong the collective effect due to interactions is compared to the self-interaction.
This concept
was introduced in the physics literature in \cite{MR84}. A precise mathematical discussion of this
length and its relevance in phase transition problems driven by diffusive effects can be found
in \cite{NO01}, \cite{NV06}.
This quantity is the so called screening length $\xi$ which is given by
\[
	\xi = \frac1{\sqrt{N R}}.
\]
It can indeed be viewed as a length scale since $d:= N^{-1/3}$ is the typical distance between particles.
Moreover,  $\xi^2$ is the inverse of the capacity density of the particles.

Heuristically the strength of the interactions is computed as follows.
The Stokes drag force for a single particle with velocity $V_s$ in a fluid at rest without external forces
 is given by the well-known formula
\[
	F_{\operatorname{St}} = 6 \pi \mu R V_s.
\]
Since the particle is inertialess, the sum of this drag force,  gravity, and buoyancy is zero.
(In \eqref{eq:Inertialess}, the buoyancy is hidden in the right hand side since the pressure contains a term $\rho_f g \cdot x$ due to the external force.) Thus,  
\[
	6 \pi \mu R V_s = \frac{4\pi}{3} R^3 (\rho_p - \rho_f) g,
\]
and therefore,
\[
	V_s = \frac{2}{9} R^2 \frac{ (\rho_p - \rho_f) g}{\mu}.
\]
We define  $\phi := N R^3$ which is the order of the volume of the particles.
Moreover, we denote
\[
	e := \frac{ (\rho_p - \rho_f) g}{\mu}.
\]
Then,
\begin{equation}
	\label{eq:SpeedOfSingleSphere}
	V_s = \frac{2}{9} \xi^2 \phi e.
\end{equation}

Now, we turn to the computation of the collective effect of the other particles
on a particle $\bar{X}_i$. To first order, the change of fluid velocity near the particle $\bar{X}_i$ due
to a particle $\bar{X}_j$ is expected to be like the fluid velocity generated by the single sphere
considered above centered at $\bar{X}_j$.
The fluid velocity corresponding to this moving single sphere  decays roughly like 
$\frac{1}{|x-\bar{X}_j|}$.
Therefore, the collective effect of the other particles is expected to be of order
\begin{equation}
	\label{eq:collectiveEffect}
	\frac{4\pi}{3} R^3 \sum_{j \neq i} \frac{1}{|\bar{X}_i - \bar{X}_j|} |e| \sim 
	\frac{R^3}{d^3} |e| \int_{\Omega} \frac{1}{|\bar{X}_i - x|} \dd x  \sim N R^3 |e| =\phi |e|,
\end{equation}
where we assumed that all particles are contained in a domain $\Omega$ with a size that is of order one.

Comparing \eqref{eq:SpeedOfSingleSphere} and \eqref{eq:collectiveEffect} suggests 
that the smaller $\xi$, the more relevant interactions become.
In \cite{JO04}, it has been rigorously proven, that in the limit $\xi \to \infty$ particle interactions are negligible.
In contrast, we want to focus on the case where interactions are relevant ($\xi$ of order one)
or even dominant ($\xi \to 0$). In this case, \eqref{eq:collectiveEffect} suggests
that the velocity of the particles (and of the fluid) is of the order of their total volume $\phi$.
Note that since we assumed that the size of the particle cloud is of order one,
the volume fraction of the particles is of the same order as their total volume.

For small volume fractions of the particles $\phi$, this justifies modeling the fluid by Stokes equations, since the Reynolds number is proportional to the velocity.
Moreover, the typical acceleration of the particles can be expected to be proportional to $\phi^2$,
which means that particle inertia are higher order terms.

\subsection{Formulation of the main result}
\label{sec:mainResult}
We consider a  sequence  of initial particle configurations indexed by $\eps$
and we assume  $N_\eps \to \infty$ and $R_\eps \to 0$ in the limit $\eps \to 0$.
Moreover, we assume $\lim_{\eps \to 0} \xi_\eps = \xi_\ast \in [0,\infty)$.

Furthermore, we consider the mass density of the particles
 \[
 	\bar{\rho}_\eps(t,\cdot) := \rho_p \sum_i \chi_{B_i(t)},
 \]
 where the particle positions depend on $\eps$. 
  The dynamics \eqref{eq:StokesBdry} implies that the
   particles are transported by the velocity field $\bar{v}_\eps$, i.e,
  \[
  	\partial_t \bar{\rho}_\eps + \bar{v}_\eps \cdot \nabla \bar{\rho}_\eps = 0.
  \]
  
The formal computation of the order of the velocities of the particles \eqref{eq:collectiveEffect}
 suggest the following rescaling.
 
\begin{equation}
	\label{eq:rescaledMass}
 	\rho_\eps(t,x) := \frac{1}{\rho_p \phi_\eps} \bar{\rho}_\eps(\frac{t}{\phi_\eps},x),
\end{equation}
and
\[
  	v_\eps(t,x) := \frac{1}{\phi_\eps} \bar{v}_\eps(\frac{t}{\phi_\eps},x).
\]
Then,
  \[
  	\partial_t {\rho}_\eps + {v}_\eps \cdot \nabla {\rho}_\eps = 0.
\]

The main result of this paper is the convergence of the dynamics to the macroscopic equation
	\begin{equation}
\label{eq:transportStokes}
\begin{aligned}
	\partial_{t} \rho + \left(\frac{2}{9} \xi_\ast^2 e  + v_\ast \right) \cdot \nabla \rho &= 0, \\
	\rho(0,\cdot) &= \rho_0, \\
	-\Delta v_\ast + \nabla p &= \rho e, \\
	\dv v_\ast &= 0, \\
	v_\ast(x) &\to 0 \quad \text{as} ~ |x| \to \infty.
\end{aligned}
\end{equation}
The convergence holds in a certain averaged sense for the rescaled
mass density $\rho_\eps$ given that the corresponding  initial mass density  ${\rho}_{\eps,0}$
converges in the same averaged sense (cf. Assumption \ref{ass:ConvergenceInitialData}). 
Moreover, we need to impose that initially, the distances between particles are bounded below 
by a fraction of their typical distance $N_\eps^{-1/3}$, 
and that the system is dilute in the sense $\phi_\eps \log N_\eps \to 0$ 
(cf. assumptions \ref{cond:particlesSeparated} - \ref{cond:screeningLength} in Section 2).

The precise statement of the main result will be given in Theorem \ref{th:main}.

\subsection{Interpretation of the macroscopic equation}

The macroscopic equation \eqref{eq:transportStokes} is a nonlinear transport equation for the 
limit averaged mass density of the particles.
The transport velocity $\frac{2}{9} \xi_\ast^2 e  + v_\ast$ consists of two parts.
The velocity $v_\ast$ is the macroscopic fluid velocity, which is given as the solution to Stokes
equations with a source term proportional to the macroscopic mass density.
The explanation for the source term is that by equation  \eqref{eq:Inertialess} every particle
induces a force inside the particle which is proportional to the mass of the particle.
However, the cloud moves faster than
the macroscopic fluid velocity by an additional $\frac{2}{9} \xi_\ast^2 e$.
This corresponds exactly to the speed of a single particle computed in \eqref{eq:SpeedOfSingleSphere}.

Let us consider the limit $\xi_\ast \to \infty$.
We note that the macroscopic fluid velocity $v_\ast$ in \eqref{eq:transportStokes}  is of order one and thus much smaller than $\frac{2}{9} \xi_\ast^2 g$. Rescaling time according to $t' = t/\xi_\ast^2$
yields in the limit $\xi_\ast \to \infty$
\begin{equation}
\label{eq:LimitForScreeningLengthToInfty}
\begin{aligned}
	\partial_{t'} \rho + \frac{2}{9} e \cdot \nabla \rho &= 0, \\
	\rho (0,\cdot) &= \rho_0. \\
\end{aligned}
\end{equation}
This means that the particles simply fall down with constant speed as single particles,
which is in accordance with the results in \cite{JO04}.

On the other hand, in the case $\xi_\ast = 0$, the self-interactions are negligible as expected.
For positive but finite $\xi_\ast$, the behavior of solutions to \eqref{eq:transportStokes}
is very similar as in the case $\xi_\ast = 0$.
Indeed, although equation \eqref{eq:transportStokes} is nonlinear, 
the only effect of  the term $\frac{2}{9} \xi_\ast^2 e$ 
is a translation velocity of the cloud (cf. Proposition \ref{pro:effectConstant}).
The reason for this is that, as a convolution operator, the solution operator of the Stokes equations
commutes with translations.

An important observation is that equation \eqref{eq:transportStokes} can be interpreted as
modeling the evolution of a fluid with variable density but fixed viscosity.
In this case, $\rho$ is the difference of density of the fluid to the density at infinity.
In particular, \eqref{eq:transportStokes} models the settling of a fluid drop surrounded by 
a fluid of larger density.

The analogy between a suspension of particles and a fluid drop has been observed
in experiments
(see e.g. \cite{PM82}, \cite{KHA84},  \cite{MMNS01}, \cite{MNG07}),
Moreover, the macroscopic equation \eqref{eq:transportStokes},
 has been obtained by formal computations and referred to as a 
 `continuum model for sedimentation' (see e.g. {\cite{Fe84}, \cite{Lu00}, \cite{MMNS01}}).

\subsection{Outline of the proof}

The main difficulty of the analysis of the dynamics \eqref{eq:StokesBdry}, \eqref{eq:Inertialess}, \eqref{eq:initialData} is that the fluid velocity $\bar{v}_\eps$ is only implicitly given. 
It satisfies Stokes 
equations but the source term is not given explicitly but only the total force on each particle
(by \eqref{eq:Inertialess}) and the constraint that the fluid velocity has to be constant
at every particle. Those constants, however, which are a priori unknown, determine the velocity
of the particles, and therefore, are the relevant quantities in order to understand the dynamics.

As an approximation for the rescaled fluid velocity $v_\eps$, we take the velocity $u_\eps$
which corresponds
to a source term that consists of a sum of forces uniformly distributed on the boundary of the particles such that \eqref{eq:Inertialess} holds (in its rescaled version). This approximation $u_\eps$ does not satisfy the constraint of constant velocity at the particles. Smallness of 
$\nabla (v_\eps - u_\eps)$ in $L^2(\IR^3)$ can be obtained by standard methods
in the limit of small volume of the particles,
but since we are interested in the values of $v_\eps$ at the particles, this is not good enough.

Therefore we use a rigorous version of the method of reflections, which gives a series 
representation of $v_\eps$.
The method of reflection is a method to express the solution operator for an elliptic problem with a
boundary consisting of several connected components in terms of a series involving 
the solution operators for the individual components. 
This method is useful, if the solution operators for those individual components are 
well understood as in the case of spheres.

A version of this method has been used in \cite{JO04} to study systems with very large screening lengths
$\xi$.
We will use the formulation of the method of reflection in the framework of orthogonal projection that has been
investigated in \cite{HV16}, where Stokes equations with Dirichlet boundary conditions
are considered. In that case, the series representation has been proven to converge if
the screening length $\xi$  is sufficiently large (i.e., if the capacity density of the particles is sufficiently small).
It turns out that in the case of the mixed boundary conditions given in \eqref{eq:StokesBdry}, \eqref{eq:Inertialess}, the method is actually convergent under milder assumption.
Indeed, the series is proven to converge to $v_\eps$ in $L^\infty (\IR^3)$ under the assumption
that particles are sufficiently separated and $\phi_\eps \log N_\eps$ is small (cf. \ref{cond:phiLogN}). 
(This assumption, which is only slightly stronger than smallness
of the particle volume $\phi_\eps$, seems to be unavoidable when using the method of reflections,
at least without additional assumption on the distribution of particles.
The $L^2$-estimates, however, suggest that smallness of the particle volume  $\phi_\eps$ should be sufficient.) Moreover, the zero order term of the representation, which is exactly $u_\eps$ from 
above, is shown to be close to $v_\eps$ in $L^\infty(\IR^3)$ in the limit $\eps \to 0$.

Replacing $v_\eps$ by $u_\eps$ is the most important step in the proof of the homogenization result.
Indeed, $u_\eps$ is given explicitly in terms of the particle positions and is close to
the solution of the Stokes equations with a source term proportional to the rescaled 
mass density of the particles. This leads to the macroscopic fluid velocity $v_\ast$ given in 
equation \eqref{eq:transportStokes}.

Another issue is whether  aggregation of particles takes place. Not only is this important to investigate
in order to rule out particle collisions, but also since particle aggregation would prevent convergence of the method of reflections. 
We will prove Lipschitz type estimates for the fluid velocity $v_\eps$, which are derived  using  again the method of reflections, to show that particle aggregation cannot take place in short times. 
As long as all the particles remain well separated, the homogenization result 
is then shown by analyzing how the mass density is transported along characteristics.
In order to prove that particle aggregation does not take place for arbitrary finite times,
we use a posteriori estimates on the averaged particle mass density provided by the macroscopic equation.

\subsection{Organization of the paper}

The remaining parts of this paper are organized as follows.

In Section 2, we give the precise assumption on the initial particle configurations,
as well as the definition of the averaged mass density and the relevant space for the convergence.
At the end, we state the main result of this paper.

In Section 3, we will prove that the dynamics \eqref{eq:StokesBdry}, \eqref{eq:Inertialess}, \eqref{eq:initialData} are well-posed until the first collision of particles.
This is done by showing that the particle velocities are uniquely determined by the particle positions
and that the corresponding function is Lipschitz continuous.

In Section 4 we study two approximations for the rescaled fluid velocity $v_\eps$, namely 
$u_\eps$ and $\tilde{u}_\eps$ which correspond to uniform force distributions on the boundary
and the interior of the particles respectively. 
Moreover, we prove $L^2$-estimates for $\nabla (u_\eps - v_eps)$.

In Section 5, we show that the series representation for $v_\eps$ provided by the method of reflections
converges in $L^\infty(\IR^3)$ provided that the particles are sufficiently separated.
Moreover we  prove that the zero order approximation,
which corresponds to the approximation $u_\eps$ studied in Section 4, converges to $v_\eps$ in $L^\infty(\IR^3)$ in the limit $\eps \to 0$.

In Section 6, we prove that the time, for which any distance between two particles is halved,
is bounded above uniformly in $\eps$. To do this, we prove a Lipschitz type estimate for $v_\eps$
where the Lipschitz constant depends on how aggregated the particles already are. 

In Section 7, we prove estimates for the difference of (approximations of) 
the microscopic fluid velocity $v_\eps$ and 
(approximations of) the macroscopic fluid velocity $v_\ast$ from equation \eqref{eq:transportStokes}.

In Section 8.1, we prove the convergence to the macroscopic equation up to
times, for which the particles remain sufficiently separated. In Section 8.2, we extend this convergence
to arbitrary times by proving that the particles actually will remain sufficiently separated.

In Section 9, we prove well-posedness of the macroscopic equation \eqref{eq:transportStokes}.

\section{Assumptions on the initial particle configuration}
\label{sec:assumptionsInitialConfiguration}

We consider a  sequence  of initial particle configurations $\{X^0_{\eps,i}\}_{1 \leq i \leq N_\eps}$ indexed by $\eps$
and we assume  $N_\eps \to \infty$ and $R_\eps \to 0$ in the limit $\eps \to 0$.

\begin{notation}
For the ease of notation, we write $X^0_i$ instead of $X^0_{\eps,i}$ in the remainder of this paper. 
We will also sometimes drop the index $\eps$
	on other quantities, in particular when $\eps$ is fixed.
\end{notation}

We impose the following constraints on the initial particle distributions.

\begin{enumerate}[label=\textbf{(A\arabic*)}, ref=(A\arabic*)]

\item
\label{cond:particlesSeparated}	
We require the distance between every pair of particles to be at least of the order of the typical distance between particles. 
More precisely, the minimal distance
\begin{equation}
	\label{eq:dMin}
	d_{\eps,\min}(0) := \min_{i \neq j} |X_i^0 - X^0_j|
\end{equation}
has to satisfy
\begin{equation}				
		N_\eps (d_{\eps,\min}(0))^3  \geq c_0 
\end{equation}
	for some constant $c_0 > 0$ independent of $\eps$.

\item
\label{cond:phiLogN}
	We define 
\begin{equation}
		 \phi_\eps := N_\eps R_\eps^3,
\end{equation}		
which is the order of volume of the particles. We require
\[
	\phi_\eps \log N_\eps \to 0 \quad \text{as} ~ \eps \to 0.
\]
	
\item 	\label{cond:screeningLength}
The screening length of the system of particles tends to some finite limit. More precisely,
\begin{equation}
		 \xi_\eps := \frac{1}{\sqrt{N_\eps R_\eps}} \to \xi_\ast \quad \text{as} ~ \eps \to 0.
\end{equation}	

%

\end{enumerate}

Note that, for a fixed $\eps >0$, the minimal distance between the particles $d_{\eps,\min} $
might change over time. Thus, an important issue for the analysis of the time evolution will be to examine, whether condition \ref{cond:particlesSeparated} is conserved over time, 
possibly with a smaller constant but uniformly in $\eps$. Therefore we introduce the following quantity.

\begin{definition}
	\label{def:Y}
	For $\eps>0$ we define 
	\[
		Y_\eps (t) := \sup_{ 0 \leq s \leq t} \frac{|X_i^0 - X_j^0|}{|X_i(t) - X_j(t)|},
	\]
	where the particle positions implicitly depend on $\eps$.
\end{definition}

\begin{remark}
\label{rem:particlesNotTouching}
Implicitly, we also assume the particles to be disjoint. Indeed,
for sufficiently small $\eps$ this is ensured by 
\ref{cond:particlesSeparated} and \ref{cond:phiLogN}.

Moreover, this is preserved up to time $t$, provided  $Y_\eps(t)$ satisfies a uniform bound for small
$\eps$:
\[
	d_{\eps,\min}(t) \geq \frac{d_{\eps,\min}(0)}{Y_\eps(t)} \geq  \frac{c_0^{1/3}}{N_\eps^{1/3} Y_\eps(t)} 
	=  \frac{c_0^{1/3}}{\phi_\eps^{1/3} Y_\eps(t)} R_\eps \geq 4 R_\eps,
\]
for all $\eps < \eps_0(t)$ small enough.
We will always assume that $\eps$ is chosen small enough, such that this is the case.
\end{remark}

Notice that in the case of strictly positive limiting screening length $\xi_\ast$, assumption \ref{cond:phiLogN} is automatically satisfied. Indeed,
	\[
		\phi_\eps \log N_\eps = N_\eps R_\eps^3 \log N_\eps \leq R_\eps (N_\eps R_\eps)^2 
		= R_\eps \xi_\eps^{-4} \to 0.
	\]

Finally, we also impose that the initial particle configurations converge in a certain averaged sense.

\begin{definition}
	\label{def:CubeAverages}
	For $\delta > 0$ we decompose $\IR^3$ up to a nullset into open disjoint cubes $Q_\delta^i$ with edge length $\delta$.
	Then, we define  $\rho_\eps^\delta$ by 
	\[
		 \rho_\eps^\delta(x) = \fint_{Q_{\delta}^i} \rho_\eps(y) \dd y \quad \text{for} ~ x \in Q_\delta^i.
	\]
	We will denote the cube containing $x$ by $Q_\delta^x$ (which is unique and exists for a.e. $x$).
\end{definition}

\begin{definition}
\label{def:X}
	Let $\beta \geq 0$. We define the norm 
	\[
		\|h\|_{X_\beta} := \sup_{x} (1+|x|^\beta)|h(x)|,
	\]
	and the space 
	\[
		X_\beta := \{ h \in L^\infty(\IR^3) \colon \|h\|_{X_\beta} < \infty\}.
	\]
\end{definition}

\begin{assumption}
\label{ass:ConvergenceInitialData}
There exists a sequence $ d_{\eps, \min}(0) \ll \delta_\eps \to 0$ and a function $\rho_0 \in X_\beta$
with $\nabla \rho \in X_\beta$ for some $\beta >2$ such that 
\[
\lim_{ \eps \to 0}	\| \rho_{\eps,0}^{\delta_\eps} - \rho_0 \|_{X_\beta} = 0.
\]
\end{assumption}

\begin{remark}
	\label{rem:orderOfDelta}
	For all $\tilde{\delta}_\eps \to 0 $, such that $\tilde{\delta}_\eps = n_\eps \delta_\eps$ for some $n_\eps \in \IN \backslash 0$,
	\[
		\| \rho_{\eps,0}^{\delta_\eps} - \rho_0 \|_{L^\infty} \to 0.
	\]
	Thus, we can assume $\delta_\eps \gg \phi_\eps$.
\end{remark}

The following theorem is the main result of this paper.
\begin{theorem}
\label{th:main}
	Assume that conditions \ref{cond:particlesSeparated} - \ref{cond:screeningLength}
	are satisfied and that the initial data $\rho_{\eps,0}$ converge to some $\rho_0$ in the sense of 
	Assumption \ref{ass:ConvergenceInitialData} with some $\delta_\eps \to 0$ and $\beta >2$.
	Then for all $T >0$, there exists $\eps_0 >0$ such that for all $\eps < \eps_0$, there exists a unique 
	solution to the dynamics \eqref{eq:StokesBdry}, \eqref{eq:Inertialess}, \eqref{eq:initialData}
 without collisions up to time $T/\phi_\eps$.
	Moreover, for all $\tilde{\delta}_\eps \to 0$ such that $\tilde{\delta}_\eps = n_\eps \delta_\eps$
	for some $n_\eps \in \IN^\ast$ with $n_\eps \to \infty$,
	\[
		\rho_\eps^{\tilde{\delta}_\eps} \to \rho  \quad \text{in} ~ L^\infty(0,T;X_\beta),
	\]
	where $\rho$ is the unique classical solution to problem \eqref{eq:transportStokes}.
\end{theorem}

\section{Well-posedness of the dynamics away from collisions}

We rewrite the dynamics \eqref{eq:StokesBdry}, \eqref{eq:Inertialess}, \eqref{eq:initialData},
absorbing the gravitational force into the pressure by defining $p(x) = q(x) - \rho_f g \cdot x$.
Moreover, we write the dynamics directly in the rescaled version dividing time and the velocities by $\phi$, as in Section \ref{sec:mainResult}. 
Furthermore, we extend the fluid velocity ${v}$ to a function defined in the whole space, 
by setting $v = V_i$ in $B_i$. Finally, we divide the PDE for $v$ by $\mu$ without renaming the pressure.
Then,

\begin{equation}
\begin{aligned}
	\label{eq:Velocity}
	\dot X_i(t) &= V_i(t),	\\
	X_i(0) &= X_i^0,	
\end{aligned}
\end{equation}
and
\begin{equation}
\label{eq:StokesBdryGivenForces}
\begin{aligned}
	-\Delta  v(t,\cdot) + \nabla q(t,\cdot) &= 0 \quad \text{in} ~ \IR^3 \backslash \bigcup_{i=1}^N \overline{B_i(t)}, \\
	\dv {v}(t,\cdot) &= 0 \quad \text{in} ~ \IR^3 ,\\
	{v}(t,\cdot) &= V_i(t,\cdot) \quad \text{in} ~ \overline{B_i(t)}, \\
	{v}(x) &\to 0  \quad \text{as} ~ |x| \to \infty, 
\end{aligned}
\end{equation}
and
\begin{equation}
	\label{eq:inertialess}
		\int_{\partial B_i} \sigma n \dd \mathcal{H}^2 = -\frac{4 \pi}{3N} \frac{(\rho_p - \rho_f)  g}{\mu} =: -F 
	 \quad \text{for all} ~ 1 \leq i \leq N, \\
\end{equation}
where the fluid stress is now $ \sigma =  \nabla {v} + (\nabla {v})^T - q I$ 
and the force $F$ is the sum of gravity and buoyancy.

\begin{notation}
We denote by $\dot{H}^1(\IR^3;\IR^3)$ the homogeneous Sobolev space, which is 
defined as the closure of $C_c^\infty(\IR^3;\IR^3)$ with respect to the $L^2$-norm of the gradient.
Moreover, we denote by $\Hdot_{\sigma}(\IR^3)$ the space of all divergence free functions
in $\dot{H}^1(\IR^3;\IR^3)$.

\end{notation}

For a fixed time, it is well known that problem \eqref{eq:StokesBdryGivenForces} has a unique weak solution $({v},q)$ in $\dot{H}^1(\IR^3;\IR^3) \times L^2(\IR^3)$
given the data of the particles, $X_i$, $V_i$,  provided the particles are not touching each other.

Moreover, for a fixed time, problem \eqref{eq:StokesBdryGivenForces}, \eqref{eq:inertialess} has a unique weak solution $v \in \Hdot(\IR^3;\IR^3)$
 in terms of 
$X_i(t)$ provided the particles are non-touching.
To see this, we fix the space positions of the particles $X_i$ and observe that, for given velocities $V_i$, the forces $G_i = -\int_{\partial B_i} \sigma n \dd \mathcal{H}^2$ are given by $A V$, where 
$A \in \IR^{N \times N}$ is a linear map.
Furthermore, $A$ is coercive, because
\begin{equation}
	\label{eq:ACoercive}
	V \cdot A V = V \cdot G = - \sum_{i=1}^N \int_{\partial B_i} V_i \cdot \sigma n \dd \mathcal{H}^2 
	=  \int_{\IR^3 \backslash \bigcup_{i=1}^N \overline{B_i}} |\nabla {v}|^2 = \| {v} \|^2_{\Hdot(\IR^3)} \geq C \| V \|^2.
\end{equation}
Hence, $A$ is invertible, which yields $V$ for prescribed $G$ and $X$.

\begin{theorem}
	For any initial configurations of 
	particles $(X_i^0)_{1\leq i\leq N}$
	such that the closed balls $\overline{B_i^0}$ are pairwise disjoint,
	there exists a time $T_\ast > 0$ such that
	the problem \eqref{eq:Velocity}, \eqref{eq:StokesBdryGivenForces}, \eqref{eq:inertialess} up to time $T_\ast$ 
	has a unique solution. Moreover, at time $T_\ast$, there exist particles $i \neq j$ such that
	$\overline{B_i} \cap \overline{B_j} \neq \emptyset$.
\end{theorem}

\begin{notation}
	Since we always consider solutions in $\Hdot_{\sigma}(\IR^3)$ to problem
	 \eqref{eq:StokesBdryGivenForces}, we will not write the condition
	 ${v}(x) \to 0$  as  $|x| \to \infty$ in the rest of the paper.
\end{notation}

\begin{proof}
	We have seen that the velocities $V_i(t)$  are uniquely determined by the
	particle positions. Hence we can write $V_i(t) = W_i(X(t))$. 
	Then, it suffices to prove that the function $W_i$ is locally Lipschitz continuous away from particle
	collisions.
	
	We can estimate the $\Hdot$-norm of the solution $v$ to problem \eqref{eq:StokesBdryGivenForces},
	 \eqref{eq:inertialess}
	brutally using \eqref{eq:ACoercive} and the definition of $F$ in  \eqref{eq:inertialess}.
	In the following a constant $C$ might depend on $R$ and $N$, which are both fixed.
	\begin{equation}
	\begin{aligned}
		\|v\|^2_{\Hdot(\IR^3)} &= \sum_i F V_i  
		\leq C \sup_i  |V_i|  \leq C \sup_i \left | \fint_{B_i} v(y) \dd y \right | \\
		& \leq  C  \sup_i \|v\|_{L^1(B_i)}  \leq C  \sup_i \|v\|_{L^6(B_i)} \|1\|_{L^\frac{6}{5}(B_i)}
		\leq C \|v\|_{\Hdot(\IR^3)}.
	\end{aligned}
	\end{equation}
	Dividing by $\|v\|_{\Hdot(\IR^3)}$ yields
	\begin{equation}
		\label{eq:brutalVelocityEstimate}
		\|v\|_{\Hdot(\IR^3)} \leq C,
	\end{equation}
	independently of the particle positions.
	
	Fix 
	particle positions with non-touching particles $(X_i)_{1\leq i\leq N}$. 
	Then, there exists $\theta > 1$ such that 
	the closed balls $\overline{B_{ 2 \theta R}(X_i)}$ are pairwise disjoint.
	Note that we can choose the same value  $\theta $ for particle positions $\tilde{X_i}$ with
	$\|X - \tilde{X}\|$ sufficiently small. Therefore, any dependencies on $\theta$ in the estimates 
	that we are going to derive do not matter for proving local Lipschitz continuity.
	Let $(\tilde{X}_i)_{1\leq i\leq N}$ be another particle configuration with
	 \[
	 	\sup_i |X_i - \tilde{X}_i| \leq \frac{(\theta - 1)R}{4} 
	 \] 
	We define a deformation $\varphi$ by
	\[
		\varphi(x) := x + \sum_i (\tilde{X}_i - X_i) \eta_i(x),
	\]
	where $\eta_i \in C_c^\infty(B_{\theta R}(X_i))$ are chosen such that
	$0 \leq \eta_i \leq 1$, $\eta_i = 1 $ in $\overline{B_i}$ and 
	\[
		|\nabla \eta_i| \leq \frac{2}{(\theta - 1)R}.
	\]
	Then, $\phi$ is a diffeomorphism and $| \nabla \varphi |, | \nabla \varphi^{-1} | \leq C$.	
	
	Consider now the solutions $v$ and $\tilde{v}$ of problem \eqref{eq:StokesBdryGivenForces}
	with particle positions $X_i$ and $\tilde{X_i}$, respectively.
	We denote the velocities in the balls $B_i$ and $\tilde{B}_i$ by $V_i$ and $\tilde{V}_i$, respectively.
	We define $u_1 := \tilde{v} \circ \varphi$.
	Then,
	\[
		|\dv u_1| \leq C \sum_i |\tilde{X}_i - X_i| |\nabla \tilde{v}(\varphi(x))| 
		\chi_{B_{\theta R}(X_i) \backslash \overline{B_i}}.
	\]
	By Lemma \ref{lem:divergenceSolution}, there exists a function 
	$u_2 \in H^1_0(\cup_i B_{\theta R_i}(X_i) \backslash \overline{B_i})$ such that
	$\dv u_2 = \dv u_1$ and 
	\[
		\| u_2\|_{\Hdot(\IR^3)} \leq  C \|\dv u_1\|_{L^2(\IR^3)}
		\leq C \|X-\tilde{X}\| \| \tilde{v}\|_{\Hdot(\IR^3)}.
	\]
	Finally, we define $u = u_1 - u_2$. 	Then, $u = \tilde V_i$ in $B_i$. Moreover, using the equation,
	that $\tilde{v}$ satisfies, we observe
	\begin{align}
		-\Delta u + \nabla p &= - \dv g \quad \text{in} ~ \IR^3 \backslash \bigcup_{i=1}^N \overline{B_i}, \\
			\dv u &= 0 \quad \text{in} ~ \IR^3,
	\end{align}
	where 
	\[
		g(x) = - \sum_i ((\tilde{X}_i - X_i) \otimes \nabla \eta_i (x) ) \nabla \tilde{v}(\varphi(x)) - \nabla u_2.
	\]
	Thus,
	\begin{equation}
		\label{eq:estimateSourceG}
		\|g\|_{L^2(\IR^3)}  \leq C \|X-\tilde{X}\| \| \tilde{v}\|_{\Hdot(\IR^3)}.
	\end{equation}
	Moreover, with $\sigma_u$ denoting the stress corresponding to $u$,
	\[
		\int_{\partial B_i} \sigma_u n \dd \mathcal{H}^2 = -F,
	\]
	where $F$ is the force defined in \eqref{eq:inertialess}.
	Defining $w:= u-v$, we then deduce that $w$ satisfies the following equation in its weak formulation
	\[
		(\nabla w, \nabla \psi ) = (g, \nabla \psi) \qquad \text{for all} \quad \psi  \in \Hdot_\sigma(\IR^3)~ \text{with} ~ \psi = \mathrm{const} ~ \text{in} ~ B_i,~ 1 \leq i \leq N.
	\]
	Testing with $\psi = w$ and using the bound for $g$ from \eqref{eq:estimateSourceG},
	and \eqref{eq:brutalVelocityEstimate} for the norm of $\tilde{v}$, we deduce
	\[
		\|w\|_{\Hdot(\IR^3)} \leq C \|X-\tilde{X}\|.
	\] 
	Since $w = V_i - \tilde{V}_i$ in $B_i$, this yields
	\[
		\|V-\tilde{V}\| \leq C \|X-\tilde{X}\|,
	\]
	which concludes the proof.
\end{proof}

The following Lemma can be found in every standard textbook on Stokes equations, e.g., in \cite{Ga11}.

\begin{lemma}
	\label{lem:divergenceSolution}
	Let $\Omega \subset \IR^3$ be a bounded domain, which is locally Lipschitz, and assume that $f \in L^2(\Omega)$ satisfies
	\[
		\int_\Omega f = 0.
	\]
	Then, there exists $u \in H^1_0(\Omega)$ such that 
	\[
		\dv u = f
	\]
	and
	\[	
		\|\nabla u\|_{L^2(\Omega)} \leq  C \|f\|_{L^2(\Omega)},
	\]
	where the constant depends only on $\Omega$.
\end{lemma}


\section{Approximations for the velocity field $v_\eps$}
	\label{sec:VariationalProblem}

The solution ${v_\eps}$ to problem \eqref{eq:StokesBdryGivenForces} satisfies 
\begin{equation}
\begin{aligned}
	-\Delta {v_\eps} + \nabla q &= h \quad \text{in} ~ \IR^3, \\
	\dv {v_\eps} &= 0 \quad \text{in} ~ \IR^3,
\end{aligned}
\end{equation}
where the force density is given by 
\begin{equation}
	\label{eq:forceDensityOfV}
 h_\eps = \sum_i h_i = \sum_i - \sigma n \delta_{\partial B_i}.  
\end{equation}
The problem is that this force density $h$ is only implicitly given by the total forces on each particle,
which is given by $F_\eps$ defined in \eqref{eq:inertialess},
and the constraint of constant velocity at every particle.

We want to replace this force density by some quantity depending only on the total force $F_\eps$
defined in \eqref{eq:inertialess}, which is minus the integral of $h_i$.
There are two convenient choices for this replacement, either a uniform force distribution in the particles, or
a uniform force distribution on their boundaries.

More precisely, we define
\begin{equation}
	\label{eq:defFi}
	{f}_i := \frac{F_\eps}{| \partial B_i|} \delta_{\partial B_i},
\end{equation}
and ${u_\eps} \in \Hdot_\sigma(\IR^3)$ to be the solution to the equation
\begin{equation}
	\label{eq:u}
	\begin{aligned}
	-\Delta {u_\eps} + \nabla p &= \sum_i {f}_i =: {f_\eps}, \\
		\dv {u_\eps} &= 0 .
	\end{aligned}
\end{equation}

Furthermore, we define
\begin{equation}
	\label{eq:defTildeFi}
	\tilde{f}_i := \frac{F_\eps}{|  B_i|} \chi_{ B_i},
\end{equation}
and $\tilde{u}_\eps \in \Hdot_\sigma(\IR^3)$ to be the solution to the equation
\begin{equation}
	\label{eq:uTilde}
	\begin{aligned}
	-\Delta \tilde{u}_\eps + \nabla p &= \sum_i \tilde{f}_i =: \tilde{f}_\eps, \\
		\dv \tilde{u}_\eps &= 0 .
	\end{aligned}
\end{equation}

Notice that both $u_\eps$ and $\tilde{u}_\eps$ satisfy the constraint of the total force acting on each particle \eqref{eq:inertialess}, but they (in general) both fail to be constant inside of the particles.

Away from the particle,  $u_\eps$ and $\tilde{u}_\eps$ can be expected to differ only little since the particles are very small.
This is expressed by the estimate in Lemma \ref{lem:estimateUUTilde}.
Nevertheless, it turns out that both approximations, ${u_\eps}$ and $ \tilde{u}_\eps$, are very useful.

The force $f_i$ concentrated on the boundary has the nice property that it
corresponds to the drag force of a particle moving the fluid. Thus, the force at each particle 
does not create any change of the velocity field inside the particle itself but only inside the other particles.
More precisely, we define $w_i \in \Hdot_\sigma(\IR^3)$ to be the solution to
\begin{equation}
	\begin{aligned}
	-\Delta {w_i} + \nabla p &= {f}_i, \\
		\dv {w_i} &= 0.
	\end{aligned}
\end{equation}
Then, it is well known that for all $x \in B_i$
\begin{equation}
\label{eq:2Over9}
	w_i(x) = \frac{F_\eps}{6 \pi R_\eps} =   \frac{2}{9 N_\eps R_\eps}  \frac{(\rho_p - \rho_f)g}{\mu}
	 = \frac{2}{9} \xi^2_\eps e,
\end{equation}
where $e \in \IR^3$ a constant independent of $\eps$.
The fact that $w_i$ is constant inside $B_i$ will prove to be very important for the estimates of the difference of $u_\eps$ and $v_\eps$.

On the other hand, choosing the force uniformly distributed inside the whole particles, we get a direct 
relation between the (rescaled) mass density of the particles $\rho_\eps$ and the approximation of the fluid velocity.
We recall from \eqref{eq:rescaledMass} the definition of $\rho_\eps$
 \[
 	\rho_\eps = \frac{1}{\phi_\eps}\sum_i \chi_{B_i}.
 \]
 Thus,
 \begin{equation}
 \label{eq:massDensityIsSource}
 	\hat{f_\eps} = \frac{F_\eps}{|B_i|} \phi_\eps \rho_\eps = e \rho_\eps.
 \end{equation}
  The dynamics \eqref{eq:Velocity}, \eqref{eq:StokesBdryGivenForces} imply that the
   particles are transported by the velocity field $v_\eps$, i.e,
  \[
  	\partial_t \rho_\eps + {v}_\eps \cdot \nabla \rho_\eps = 0.
  \]
Formally replacing $v_\eps$ by $\tilde{u}_\eps$ and using \eqref{eq:massDensityIsSource} yields
\begin{equation}
	\begin{aligned}
	\partial_t \rho_\eps + \tilde{u}_\eps \cdot \nabla \rho_\eps 
	&= (\tilde{u}_\eps - v_\eps) \cdot \nabla \rho_\eps, \\
	-\Delta \tilde{u}_\eps + \nabla p &= e \rho_\eps, \\
		\dv \tilde{u}_\eps &= 0 .
	\end{aligned}
\end{equation}
This  almost looks  like the macroscopic equation \eqref{eq:transportStokes}.
As explained in the introduction, the missing term $\frac{2}{9}\xi_\ast^2 e$
appears in the limit $\eps \to 0$ since the particles move faster than the macroscopic fluid velocity. 
Thus, it formally remains to prove that the term $(\tilde{u}_\eps - v_\eps) \cdot \nabla \rho_\eps$
vanishes in the limit $\eps \to 0$.

\subsection{Estimates for $u_\eps - v_\eps$ in $\Hdot(\IR^3)$}
\label{sec:L^2Estimates}
In order to control the motion of the particles, we need estimates of $u_\eps-v_\eps$ in $L^\infty$.
Those will be shown using the method of reflections in Section \ref{sec:LInfty}.
There, we will also rely on standard methods exploiting the structure of the linear PDEs that
$u_\eps$ and $v_\eps$ solve. In this subsection, we will explain this in detail and prove an 
$L^2$-estimate for $\nabla(u_\eps-v_\eps)$.
Since we consider fixed $\eps$, we will not write the index in the following.

It is interesting to notice that both ${u}$ and ${v}$ are solutions to variational problems.
We define
\[
	E(w) := \frac{1}{2} \int_{\IR^3} |\nabla w|^2 \dd x - \sum_i F \fint_{B_i} w \dd y .
\]
Then, $u$ is the minimizer of $E$ in $\Hdot_\sigma(\IR^3)$.
Moreover, $v$ is the minimizer of $E$ in the subspace
\[
	W := \{w \in \Hdot_\sigma(\IR^3)\colon w = \mathrm{const} ~ \text{in} ~ B_i,~ 1 \leq i \leq N \}.
\]

In particular, $v$ is the orthogonal projection from $\Hdot_\sigma(\IR^3)$ to $W$.
Indeed, let $w \in W$. Then,
\[	
	\langle u - v, w \rangle = \sum_i \langle f_i - h_i, w \rangle = 0.
\]

In particular, $\| u - v \|_{\Hdot(\IR^3)} \leq \| u - w \|_{\Hdot(\IR^3)}$ for all $w \in W$.
We will exploit this by choosing $w$ in a smart way in order to get an estimate for $u-v$.
For this we need the following lemmas. The first one is a standard extension estimate.

\begin{lemma}
	\label{lem:extest}
	For $ r>0 $ and $ x \in \IR^3$,  
	let $ H_r := \left\{ u \in  H^1_\sigma(B_r(x)) \colon \int_{ B_r(x)} u = 0 \right\} $. 
	Then, for all $ r >0 $, there exists an extension operator $ E_r \colon H_r \to H^1_{\sigma,0}(B_{2r}(x)) $ such that
	\begin{equation}
		\label{eq:extest}
		\| \nabla E_r u \|_{L^2(B_{2r}(x))} \leq C \| \nabla u \|_{L^2(B_r(x))} \qquad \text{for all} \quad u \in H_r,
	\end{equation}
	where the constant $ C $ is independent of $ r $.
\end{lemma}

\begin{remark}
	An analogous statement holds for $H_r$ replaced by $\left\{ u \in  H^1_\sigma(B_r(x)) \colon \int_{\partial B_r(x)} u = 0 \right\}$.
\end{remark}

\begin{proof}
	For $ r = 1 $, let $ E_1 \colon H^1(B_1(x)) \to H^1_{\sigma,0}(B_2(x)) $ be a continuous extension operator.
	Then, by the Poincaré inequality in $ H_1 $, we get for all $ u \in H_1 $
	\begin{equation}
		\| \nabla E_1 u \|_{L^2(B_{2}(x))} \leq \| E_1 u \|_{H^1(B_{2}(x))} \leq C \| u \|_{H^1(B_1(x))}
		\leq C \| \nabla u \|_{L^2(B_1(x))}
	\label{eq:extest1}
	\end{equation}
	The assertion for general $ r > 0 $ follows from scaling by defining 
	$ (E_r)u(x) := (E_1 u_r)({\frac{x}{r}}) $ where $ u_s(x) := u(sx) $.
\end{proof}

In many of the estimates in this paper, terms like $\sum_{i \neq j} \frac {1}{|X_i - X_j|^k}$ 
for $k = 2,3$ appear. In the next lemma, we prove an estimate for those quantities in terms
of $Y$ from Definition \ref{def:Y} by approximating the sum by an integral.

\begin{lemma}
	\label{lem:brutalEstimatesSums}
	There exists a constant $C_\ast$ which depends only on $c_0$ from assumption \ref{cond:particlesSeparated} such that
	\begin{equation}
		\label{eq:brutalEstimateAlpha}
		\sup_j\frac{1}{N}  \sum_{i \neq j} \frac {1}{|X_i - X_j|^2} \leq C_\ast Y^3,
	\end{equation}
	and 
	\begin{equation}
		\label{eq:brutalEstimateDelta}
		\sup_j \frac{\phi}{N} \sum_{i \neq j} \frac {1}{|X_i - X_j|^3} 
		\leq C_\ast \phi Y^3 (\log (N) + \log(Y)).
	\end{equation}
\end{lemma}

\begin{proof}
We define $\psi \colon \IR^3 \to \IR$ by
	\[
		\psi = d_{\min}^{-3} \sum_i \chi_{B_{d_{\min}(t)/2}(X_i)},
	\]
	where $d_{min}$ denotes again the minimal particle distance.
	By Definition \ref{def:Y} and assumption \ref{cond:particlesSeparated}, we have
	\[
		\|\psi\|_{L^\infty(\IR^3)} \leq d_{\min}^{-3} \leq C_\ast N Y^3,
	\]
	and 
	\[
		\|\psi\|_{L^1(\IR^3)} = C N.
	\]
	Thus, for all particles $j$,
	\begin{equation}
		\label{eq:brutalEstimateAlpha1}
	\begin{aligned}
		\frac{1}{N} \sum_{i \neq j} \frac {1}{|X_i - X_j|^2} 
		&\leq   \frac{C}{N} \int_{\IR^3} \frac {\psi(y)}{|y  - X_j|^2} \dd y \\
		& \leq \frac{C}{N} \left( \int_{\IR^3 \backslash B_1(X_j)} \frac {\psi(y)}{|y  - X_j|^2} \dd y
			+ \int_{ B_1(X_j)} \frac {\psi(y)}{|y  - X_j|^2} \dd y \right)\\
		&\leq  \frac{C}{N} \left(\|\psi\|_{L^1(\IR^3)} + \|\psi\|_{L^\infty(\IR^3)}\right) \\
		&\leq C_\ast Y^3,
	\end{aligned}
	\end{equation}
	where we used in the last step that $Y \geq 1$. This proves \eqref{eq:brutalEstimateAlpha}
	
	To show \eqref{eq:brutalEstimateDelta}, we estimate for any $j$
	\begin{align}
		\frac{\phi}{N} \sum_{i \neq j} \frac {1}{|X_i - X_j|^3} 
		&\leq   \frac{C \phi}{N} \int_{\IR^3 \backslash B_{d_{\min}(t)/2}(X_j)} \frac {\psi(y)}{|y  - X_j|^3} \dd y \\
		& \leq \frac{C\phi}{N} \left( \int_{\IR^3 \backslash B_1(X_j)} \frac {\psi(y)}{|y  - X_j|^3} \dd y
			+ \int_{ B_1(X_j) \backslash B_{d_{\min}(t)/2}(X_j)} \frac {\psi(y)}{|y  - X_j|^3} \dd y \right)\\
		&\leq  \frac{C\phi}{N} \left(\|\psi\|_{L^1(\IR^3)} +
		 \|\psi\|_{L^\infty(\IR^3)} \log \left( \frac{1}{d_{\min}(t)} \right)\right) \\
		&\leq C_\ast \phi Y^3 \left(\log \left( \frac{1}{d_{\min}(t)} \right) + \log(Y)\right) \\
		& \leq C_\ast \phi Y^3 (\log (N) + \log(Y)).
	\end{align}
\end{proof}

\begin{remark}
	By splitting the integral in \eqref{eq:brutalEstimateAlpha1} with $B_r(X_j)$ instead of $B_1(X_j)$, 
	one can choose the optimal $r$ to find
		\begin{equation}
		\sup_j\frac{1}{N}  \sum_{i \neq j} \frac {1}{|X_i - X_j|^2} \leq C_\ast Y^2,
	\end{equation}
\end{remark}

\begin{lemma}
	\label{lem:nablaU}
	The function $u$ defined in \eqref{eq:u} satisfies
	\[
		\|\nabla u \|_{L^\infty(\cup_i B_i)} 
		\leq \frac{C}{N} \sup_i \sum_{j\neq i} \frac{1}{|X_i - X_j|^2}.
	\]
\end{lemma}

\begin{proof}
	 We have $u = \sum_i w_i$, where $w_i = S f_i$ with $f_i$ defined as in \eqref{eq:defFi} and $S \colon \dot{H}^{-1}(\IR^3) \to \Hdot_\sigma(\IR^3)$
the solution operator for Stokes equations. 
 Recall from \eqref{eq:2Over9} that $\nabla w_i = 0 $ in $B_i$.
 The solution operator $S$ can be written as the convolution with the fundamental solution, also known as the Oseen tensor,
\begin{equation}
	\label{eq:Oseen}
	\Phi(x) = \frac{1}{8\pi}\left( \frac{1}{|x|} + \frac{x \otimes x}{|x|^3} \right).
\end{equation}
 Thus, we observe 
 for all particles $i \neq j$
 \[
 	\| \nabla u_j \|_{L^\infty(B_i)} \leq C|F| \frac{1}{|X_i - X_j|^2 - 2R|X_i - X_j|} \leq C |F| 
 	\frac{1}{|X_i - X_j|^2}.
 \]
 Thus, recalling $|F| = \frac{C}{N}$ from \eqref{eq:inertialess}, 
  deduce
\[
 	\|\nabla u \|_{L^\infty(B_i)} = \big\| \sum_{j \neq i} \nabla u_j\big\|_{L^\infty(B_i)} 
 	\leq \frac{C}{N}\sum_{j\neq i} \frac{1}{|X_i - X_j|^2}.  \qedhere
\]
\end{proof}

\begin{proposition}
\label{pro:goodApproximation}
There exists a constant $C_\ast < \infty$ which depends only on $c_0$ from assumption
\ref{cond:particlesSeparated} such that
\begin{equation}
\label{eq:estimatedByVolumeFraction}
	\| u - v \|^2_{\Hdot(\IR^3)} \leq \sum_i C \| \nabla u \|^2_{L^2(B_i)} 
	\leq C_\ast \phi Y^3.
\end{equation}
\end{proposition}

\begin{proof}
For each $1 \leq i \leq N$, we can find $ w_i \in H^1_{0,\sigma}(B_{2R}(X_i))$ with 
 $\| \nabla w_i \|_{L^2(\IR^3)} \leq C \| \nabla u \|_{L^2(B_i)}$ and
 $w_i = u - (u)_i$ in $B_i$, where $ (u)_i := \fint_{B_i} u \dd x$.
Since the balls $B_{2R}(X_i)$ are assumed to be disjoint by Remark \ref{rem:particlesNotTouching}, we obtain $w := u - \sum_i w_i \in W$.
 Hence, by Lemma \ref{lem:extest},
 \begin{equation}
 \label{eq:projectionEstimate}
 	\| u - v\|^2_{\Hdot(\IR^3)}  \leq \| u - w \|^2_{\Hdot(\IR^3)} = \big\|\sum_i w_i \big\|^2_{\Hdot(\IR^3)} 
 	\leq \sum_i C \| \nabla u \|^2_{L^2(B_i)}.
 \end{equation}

 Hence, combining Lemma \ref{lem:nablaU} and Lemma \ref{lem:brutalEstimatesSums} yields
 \begin{equation}
 	 \| \nabla u \|^2_{L^2(B_i)} \leq \big\| \sum_{j \neq i} \nabla u_j\big\|^2_{L^2(B_i)} 
 	 \leq  C_\ast  R^3 Y^3 
 \label{eq:nablauInL2}
\end{equation}

Summing over all the particles and using \eqref{eq:projectionEstimate} and \eqref{eq:nablauInL2} yields
\[
	\| u - v \|^2_{\Hdot(\IR^3)} \leq \sum_i C \| \nabla u \|^2_{L^2(B_i)} \leq C_\ast N R^3 = C_\ast \phi.  \qedhere
\]
\end{proof}

\section{Estimates for $u_\eps-v_\eps$ in $L^\infty$ by the method of reflections}
\label{sec:LInfty}

In this section, we prove smallness of $v - u$ (again, we drop the index $\eps$) in $L^\infty(\IR^3)$, 
stated in Proposition \ref{pro:LInftyEstimateSeries}. We use the method of reflections
in the framework of orthogonal projections that has been investigated in \cite{HV16}.
As we will see below, the method has better convergence properties for the  problem \eqref{eq:StokesBdryGivenForces} that $v$ solves than for the Stokes equations with Dirichlet boundary conditions that have been studied in \cite{HV16}.
Indeed, for the latter the method only converges provided that the electrostatic capacity $\mu = \xi^{-2}$ is sufficiently 
small.
In \cite{HV16} this problem has been overcome by a suitable resummation procedure.

In the case at hand, however, we will see below, that the higher order terms  are associated to
force densities that are dipoles in the particles instead of monopoles. This makes the method convergent 
if 
\[
		\sup_j \frac{\phi}{N} \sum_{i \neq j} \frac {1}{|X_i - X_j|^3}
\]  
is sufficiently small. At time $t=0$, this follows from \ref{cond:particlesSeparated} and \ref{cond:phiLogN}.
For later times, we have to check that particles do not come too close.

We define 
\[
	W_i = \left\{ w \in \Hdot_\sigma(\IR^3) \colon w=\mathrm{const} ~ \text{in} ~ B_i \right\}.
\]
Let $P_i$ be the orthogonal projection from $ \Hdot_\sigma(\IR^3) $ to $W_i$ and $Q_i = 1 - P_i$.
We observe
\begin{equation}
	\label{eq:characterizationWPerp}
	W_i^\perp = \left\{ w \in \Hdot_\sigma(\IR^3) \colon \exists p \in L^2(\IR^3) ~ 
	 -\Delta w + \nabla p = 0  ~ \text{in} ~  \IR^3 \backslash \overline{B_i},
	~ \int_{\IR^3}  - \Delta w + \nabla p = 0 \right\}.
\end{equation}
Here, the first condition has to be interpreted in the weak sense. 
It  is satisfied for every $ w \in W_i^\perp$  since $\Hdot_{\sigma,0}(\IR^3 \backslash B_i) \subset W_i$.
Using the first condition, the second condition simply means
$\langle -\Delta w + \nabla p, \psi \rangle_{\dot{H}^{-1},\dot{H}^{1}} = 0 $ for all $\psi \in \Hdot(\IR^3)$ with $\psi = 1$ in 
$B_i$, and this follows directly from the definition of $W_i$.

The physical interpretation of the characterization \eqref{eq:characterizationWPerp} is that the force densities corresponding to functions in $ W_i^\perp$
 are dipoles in $\overline{B_i}$.

The method of reflections can now be stated as follows. 
Recall from the definition of $u$ in \eqref{eq:u} that $u$ satisfies the constraint of the total force acting on each particle \eqref{eq:inertialess} but fails to be constant inside of the particles.

By definition of the space $W_i$, $P_i u = (1-Q_i)u$ is constant in $B_i$. Moreover,
since $Q_i$ is a dipole, $(1-Q_i)u$ still satisfies \eqref{eq:inertialess}. Since such
a correction is needed for every particle, one makes the ansatz
\[
	u_1 = (1 - \sum_i Q_i) u.
\]
However, for every particle $i$, the term $\sum_{j \neq i} Q_j u$ again destroys the property of
constant velocity in $B_i$. 
Therefore, one repeats adding those correction terms. This leads to
\[
	 u_k  = \left(1 - \sum_i Q_i \right)^k u.
\]

For the proof that $u_k$ converges to $v$, we need the following lemmas.
Lemma \ref{lem:projectionInsideParticle} ensures that $(1 - Q_i) u_k = P_i u_k$  
does not differ too much from $u_k$ inside particle $i$.
Lemma \ref{lem:pCorrelationest} is used to exploit that $Q_j u_k$ is a dipole potential,
and therefore decays quickly.
Together, this yields $(1 - \sum_i Q_i) u_k \approx u_k$ inside of the particles.

\begin{lemma}
\label{lem:projectionInsideParticle}
	Let $w \in  \Hdot_\sigma(\IR^3) $. Then, $ P_i w = (w)_i$ in $B_i$, where $(w)_i = \fint_{\partial B_i} w$.
\end{lemma}
\begin{proof}
	Let $\psi_0 \in \IR^3$ and define $ \psi \in \Hdot_\sigma(\IR^3)$ to be the solution to
	\begin{align}
		-\Delta \psi + \nabla p &= 0 \quad \text{in} ~ \IR^3 \backslash \overline{B_i}, \\
		\psi &= \psi_0 \quad \text{in} ~  B_i.
	\end{align}
	In other words, $\psi$ is the velocity field corresponding to a moving single sphere without external forces.
	Hence, as it is well known, 
	\[
	-\Delta \psi + \nabla p	=  \frac{3}{2R} \psi_0 \delta_{\partial B_i}.
	\]
	Furthermore, $\psi \in W_i$, and hence,
	\[
		0 = (w - P_i w, \psi)_{ \Hdot(\IR^3)} = \langle w - P_i w, -\Delta \psi \rangle 
		= \frac{3}{2R} \psi_0 \cdot \int_{\partial B_i} w - P_i w \dd \mathcal{H}^2.
	\]
	Since $\psi_0$ was arbitrary, we deduce 
	\[
		\int_{\partial B_i} w - P_i w \dd \mathcal{H}^2 = 0,
	\]
	and the assertion follows.
\end{proof}

\begin{lemma}
	\label{lem:pCorrelationest}
	Assume $ f \in \dot{H}^{-1}_\sigma(\IR^3) $ is supported in 
	$ \overline{B_i}$ and $\int_{\IR^3} f = 0$, i.e., $S f \in W_i^\perp$, where $S$ is the solution operator for the Stokes equations. Then, for all
	 $x \in \IR^3 \backslash B_{2R}(X_i)$,
	\begin{equation}
		\label{eq:dipoleEstimate}
		|(S f)(x)| \leq C \frac{R^\frac{3}{2}}{|x-X_i|^2} \|f\|_{\dot{H}^{-1}(\IR^3)},
	\end{equation}
	and
	\begin{equation}
		\label{eq:dipoleEstimateGradient}
		|\nabla (S f)(x)| \leq C \frac{R^\frac{3}{2}}{|x-X_i|^3} \|f\|_{\dot{H}^{-1}(\IR^3)}.
	\end{equation}
\end{lemma}

\begin{proof}	

	We denote again by $\Phi$ the Oseen tensor \eqref{eq:Oseen}. Then, 
	\begin{equation}
	\begin{aligned}
		|(S f) (x)| &= |(\Phi \ast f)(x)| = |((\Phi - (\Phi)_{x -X_i,2R}) \ast f)(z)| \\ 
		&= |(E(\Phi - (\Phi)_{X_i -X_j,2R}) \ast f)(z)| \\
		&\leq \|f\|_{\dot{H}^{-1}(\IR^3)} \|(E(\Phi - (\Phi)_{x-X_i,2R}) \|_{\dot{H}^1(\IR^3)},
	\end{aligned}
	\label{eq:pPtw1}
	\end{equation}
	where 
	\[
		(\Phi)_{x -X_i,R} = \fint_{B_{R}(x -X_i)} \Phi(y) \dd y,
	\]
	and $E(\Phi - (\Phi)_{x -X_i,R})$ is any divergence free extension of the restriction of $\Phi - (\Phi)_{x -X_i,R}$ to 
	$B_{R}(x -X_i)$.
	By Lemma \ref{lem:extest}, we can choose this extension in such a way that 
	\[
		\|(E(\Phi - (\Phi)_{x -X_i,R}) \|_{\dot{H}^1(\IR^3)} \leq C \|\nabla \Phi \|_{L^2{B_{R}(x -X_i)}} 
		= C \frac{R^\frac{3}{2}}{|x -X_i|^2}.	
	\]
	This establishes estimate \eqref{eq:dipoleEstimate}. Estimate \eqref{eq:dipoleEstimateGradient} is proven analogously.
\end{proof}

\begin{lemma}
	\label{lem:pointsOutsideParticles}
	Let $y \in \IR^3$ and let $X_i$ be a particle that has minimal distance to $y$, i.e.,
	\[
		|y-X_i| \leq |y - X_j| \qquad \text{for all} \quad 1 \leq j \leq N. 
	\]
	Then,
	\[
		|y - X_j| \geq \frac{1}{2} |X_i-X_j| \qquad \text{for all} \quad 1 \leq j \leq N.
	\]
	In particular, for $k = 1,2$,
	\[
		\sum_{j \neq i} \frac{1}{|y - X_j|^k} \leq C \sum_{j \neq i} \frac{1}{|X_i - X_j|^k}.
	\]
\end{lemma}

\begin{proof}
	We consider two cases.
	
	\emph{Case 1.} 
	\[
		|X_j - X_i| \leq 2 |y - X_i|.
	\]
	Then,
	\[
		|y - X_j| \geq |y - X_i| \geq \frac{1}{2} |X_i-X_j|.
	\]
	
	\emph{Case 2.} 
	\[
		|X_j - X_i| \geq 2 |y - X_i|.
	\]
	Then,
	\[
		|X_j -X_i| \leq |y - X_i| + |y - X_j| \leq \frac{1}{2}|X_j - X_i| + |y - X_j|,
	\]
	and the assertion follows.
\end{proof}

The proof of the following maximum modulus estimate for solutions to Dirichlet problems of the Stokes equations
can be found in \cite{MRS99}.

\begin{lemma}
\label{lem:MaxModEstimate}
	Let $\Omega \subset \IR^3$ be a bounded or exterior domain and assume that
	 $g \in L^\infty(\IR^3)$ satisfies 
	 \[
	 	\int_{S} g \cdot \nu = 0
	 \]
	 for every connected component $S \subset \partial \Omega$.
	 Then, the unique solution $ u \in \Hdot(\IR^3)$ of the Dirichlet problem
	 \begin{align}
	 	-\Delta u + \nabla p &= 0 \quad \text{in} ~ \Omega, \\
	 	u &= g \quad \text{on} ~ \partial \Omega
	 \end{align}
	 satisfies
	 \[
	 	\| u \|_{L^\infty(\Omega)} \leq C \| g \|_{L^\infty(\partial \Omega)},
	 \]
	 where the constant $C$ depends only on $\Omega$.
\end{lemma}

\begin{remark}
	Clearly, the constant $C$ in the above statement is invariant under scaling of the domain.
	 In fact, we will only apply the above lemma for $\Omega$ being the exterior of a ball.
\end{remark}

\begin{proposition}
	\label{pro:LInftyEstimateSeries}
	We define
	\[
		u_k  = \left(1 - \sum_i Q_i \right)^k u.
	\]
	Then, for all particles $j$ and all $y \not \in B_{2R}(X_j)$,
		\begin{equation}
		\label{eq:dipoleDecay2}
		| Q_j u_k(y) | 
		\leq C  \frac{R^3}{|X_j - y|^2} \| \nabla u_{k} \|_{L^\infty(B_j)},
	\end{equation}
	and
	\begin{equation}
		\label{eq:dipoleDecay1}
		| \nabla Q_j u_k(y) | 
		\leq C  \frac{R^3}{|X_j - y|^3} \| \nabla u_{k} \|_{L^\infty(B_j)}.
	\end{equation}

	Furthermore, there exists a constant $\delta > 0$ with the following property.
	Assume that
	\[
		\sup_j \frac{\phi}{N} \sum_{i \neq j} \frac {1}{|X_i - X_j|^3} \leq \delta,
	\] 
	and assume there is some constant $\alpha < \infty$ such that
	\begin{equation}
		\label{eq:alpha}
		\sup_j\frac{1}{N}  \sum_{i \neq j} \frac {1}{|X_i - X_j|^2} \leq \alpha.
	\end{equation}
	Then,
	\begin{equation}
		\label{eq:NablaUTilde}
		\| \nabla u_k \|_{L^\infty(B_i)}  \leq  \alpha (C \delta)^k,
	\end{equation}
	and
	\[
		u_k \to v \quad \text{in}~ \Hdot(\IR^3).
	\]
	Moreover, the convergence also holds in $L^\infty(\IR^3)$, and we have 
	\begin{equation}
		\label{eq:uVLInfty}
		\|u-v\|_{L^\infty(\IR^3)} \leq C \alpha (\alpha \phi +  R).
	\end{equation}
\end{proposition}

\begin{proof}
By Lemma \ref{lem:pCorrelationest}, we have for all particles $j$ and all $y \not \in B_{2R}(X_j)$
	\begin{equation}
		\label{eq:seriesConvergent2}
		| Q_j u_{k}(y) | 
		\leq C \frac{R^{\frac{3}{2}}}{|X_j - y|^2} \| Q_j u_{k} \|_{\Hdot}.
	\end{equation}
	By \eqref{eq:characterizationWPerp}, $\supp \Delta Q_j u_{k} \subset \overline{B_i}$ as a function in $\dot{H}^{-1}_\sigma(\IR^3)$.
	Therefore, $Q_j u_{k-1}$ is the function of minimal norm in $\Hdot_\sigma(\IR^3)$ that coincides with
	$Q u_{k-1}$ in $\overline{B_i}$. Using Lemma \ref{lem:projectionInsideParticle} and Lemma \ref{lem:extest}, we deduce
	\begin{equation}
		\label{eq:seriesConvergent3}
		\| Q_j u_{k} \|_{\Hdot} \leq C \| \nabla u_{k} \|_{L^2(B_j)} 
		\leq C R^{\frac{3}{2}} \| \nabla u_{k} \|_{L^\infty(B_j)}.
	\end{equation}
	Combining \eqref{eq:seriesConvergent2} and \eqref{eq:seriesConvergent3} yields \eqref{eq:dipoleDecay1}.
	Estimate \eqref{eq:dipoleDecay2} is proven analogously.

	We claim that $v$ is orthogonal projection  of $u_k$ to $W$ for all $k \in \IN$.
	In Section \ref{sec:L^2Estimates}, we have seen that $v$ is the orthogonal projection
	of $u$ to $W$. Therefore, it suffices to observe 
	that $Q_i w$ lies in the orthogonal complement of $W$ for any $w \in \Hdot_\sigma(\IR^3)$.
	By definition, $Q_i w$ lies in the orthogonal complement of $W_i$. Since $W \subset W_i$ this implies $Q_i w \in W^\perp$.
	
	Now it follows analogously as we have obtained \eqref{eq:projectionEstimate}
	\begin{equation}
	\label{eq:seriesConvergent}
		\| v - u_k \|^2_{\Hdot(\IR^3)} \leq \sum_i C \| \nabla u_k \|^2_{L^2(B_i)} 
		\leq \sum_i C  R^3 \| \nabla u_{k} \|^2_{L^\infty(B_i)}.
	\end{equation}
	In $B_i$, we have for $k \geq 1$
	\begin{equation}
		\label{eq:seriesConvergent1}
		 \nabla u_k  = \nabla(u_{k-1} - \sum_j Q_j  u_{k-1}) = \sum_{j \neq i} \nabla Q_j u_{k-1}
	\end{equation}
	since $ u_{k-1} - Q_i u_{k-1} = P_i u_{k-1} \in W_i$ is constant in $B_i$.
	
	Thus, 
	\begin{equation}
	\begin{aligned}
		\| \nabla u_k \|_{L^\infty(B_i)} &\leq C \sum_{j \neq i} \frac{R^3}{|X_i - X_j|^3} \| \nabla u_{k-1} \|_{L^\infty(B_j)} \\
		& \leq C \delta \| \nabla u_{k-1} \|_{L^\infty(\cup B_j)}.
	\end{aligned}
	\label{eq:Linftykey}
	\end{equation}
	By Lemma \ref{lem:nablaU} and estimate \eqref{eq:alpha},
	\[
		\| \nabla u \|_{L^\infty(\cup B_j)} \leq C \alpha.
	\]
	 Combining this with \eqref{eq:Linftykey} yields \eqref{eq:NablaUTilde}. 
	Hence, by the estimate \eqref{eq:seriesConvergent}, the series $u_k$ converges to $v$ in 	
	$\dot{H}^1(\IR^3)$, provided $\delta < 1/C$.
	
	To prove convergence in $L^\infty(\IR^3)$ we choose for any fixed $x \in \IR^3$  a particle $X_i$
	which has minimal distance to $x$. 
	We note that Lemma \ref{lem:pointsOutsideParticles} implies
	\[
		\sum_{j \neq i} \frac{1}{|x - X_j|^2} \leq C \alpha.
	\] 
	Application of Lemma \ref{lem:projectionInsideParticle} 
	and Lemma \ref{lem:MaxModEstimate} for particle $i$, and  Lemma \ref{lem:pCorrelationest} for the others,
	using also \eqref{eq:seriesConvergent3}, yields
	\begin{equation}
	\label{eq:LInftyEstimateInParticle}
	\begin{aligned}
		| u_{k+1}(x) - u_k(x)| &= |\sum_j Q_j u_k(x)| \\
		&\leq \sum_{j \neq i}|  Q_j u_k(x) |  + | Q_i u_k(x)| \\
		&\leq \sum_{j \neq i} C \frac{R^\frac{3}{2}}{(x-X_i)^2} \| Q_j u_k \|_{\dot{H}^{1}(\IR^3)}  + C \| u_k - (u_k)_i \|_{L^\infty(B_i)} \\
		&\leq C \alpha R^3 N \| \nabla u_{k} \|_{L^\infty(\cup B_j)} +  C R \| \nabla u_{k} \|_{L^\infty(B_i)} \\
		&= C \alpha ( \alpha \phi +  R) (C \delta)^k.
	\end{aligned}
	\end{equation}	
	Therefore, $u_k - v$ converges to zero (in the limit $k \to \infty$ for a fixed particle configuration) in $L^\infty(\IR^3)$.

	In order to prove \eqref{eq:uVLInfty}, we use \eqref{eq:Linftykey} and \eqref{eq:LInftyEstimateInParticle} to estimate
	\begin{align}
		\| u - v\|_{L^\infty(\IR^3)} &\leq \sum_{k=0}^\infty \| u_{k+1} - u_k \|_{L^\infty(\IR^3)} \\
		& \leq \sum_{k=0}^\infty C ( \alpha \phi +  R) \| \nabla u_{k} \|_{L^\infty(\cup B_j)} \\
		&\leq C \alpha( \alpha \phi +  R) \sum_{k=0}^\infty(C \delta)^k \leq C \alpha( \alpha \phi +  R). \qedhere
	\end{align}
\end{proof}

\section{Estimates for the particle distances}
\label{sec:collisions}

In order to apply the method of reflections we have seen in Proposition \ref{pro:LInftyEstimateSeries}
that we need to control the terms
	\[
		\sup_j\frac{1}{N}  \sum_{i \neq j} \frac {1}{|X_i - X_j|^k},
	\] 
for $k = 2,3$. To achieve this, we want to estimate the quantity $Y_\eps$ from Definition \ref{def:Y},
in order to apply Lemma \ref{lem:brutalEstimatesSums}. More precisely, we show in Propsition \ref{pro:particlesRemainSeparated} that, starting at time $T_0$, the time $\theta$ needed for two particles to halve their distance is bounded from below. For sufficiently small $\eps$, this bound on $\theta$ depends only on $Y_\eps(T_0)$. Thus, if we have estimates for $Y_\eps(T_0)$ uniformly in $\eps$, we can also bound $Y_\eps(T_0 + \theta)$ uniformly in $\eps$. 

This a priori estimate enables us to prove the main theorem for small times 
(cf. Theorem \ref{th:convergenceMacro}). However, it does not rule out that $Y_\eps$ blows up in finite time.
Therefore, we prove an a posteriori estimate on $Y_\eps$ in Section \ref{sec:ConvergenceArbitraryTimes}.

In order to control the particle distances, we need to estimate their relative velocities, which are 
provided by Lipschitz type estimates for $v_\eps$. First, in Lemma \ref{lem:OszillationsOfVelocity}, we prove such an estimate for the approximated
fluid velocity $u_\eps$. Then, in Lemma \ref{lem:OszillationsOfVelocityApproximation}, we again use the method of reflections to get the estimate for $v_\eps$ as well.

\begin{lemma}
	\label{lem:OszillationsOfVelocity}
	Assume there is some constant $\alpha < \infty$ such that
	\[
		\sup_j\frac{1}{N}  \sum_{i \neq j} \frac {1}{|X_i - X_j|^2} \leq \alpha.
	\] 
	Then, 

	for all particles $i,j$ 
	and all $h \in \overline {B_R(0)} \subset \IR^3$,
	\begin{equation}
		\label{eq:lemOszPart2}
		|u_ (X_i +h) - u_(X_j+j)| \leq C \alpha |X_i-X_j|.
	\end{equation}
\end{lemma}

\begin{proof}	
	By definition, we can write
	\[
		u = \sum_i w_i,
	\]
	where $w_i$ are the solutions to
\begin{equation}
\begin{aligned}
	-\Delta w_i + \nabla p &= \frac{F}{\partial B_i|} \delta_{\partial B_i}   \quad \text{in} ~ \IR^3, \\
	\dv w_i &= 0 \quad \text{in} ~ \IR^3.
\end{aligned}
\end{equation}
Then, for $x \not\in B_R(X_i)$, 
\begin{align}
	|\nabla w_i(x)| &\leq C \frac{1}{N|x-X_i|^2}.
\end{align}

We observe that for particles $i$ and $j$ 
\begin{equation}
	\label{eq:sym1}
	 w_i (X_i +h)= w(h) = w_j (X_j+h).
\end{equation}
Moreover, using symmetry of $w$,
\begin{equation}
	\label{eq:sym2}
\begin{aligned}
	|w_j (X_i+h) - w_i (X_j+h)| &= |w(X_i - X_j + h) - w(X_j - X_i + h)| \\
	&= |w(X_i - X_j + h) - w(X_i - X_j - h)| \\
	&\leq C |\nabla w(X_i - X_j)| |h| 	\\
	&\leq C \frac{R}{N |X_i - X_j|^2} \\
	&\leq C \frac{|X_i - X_j|}{N |X_i - X_j|^2}.
\end{aligned}
\end{equation}

Let us denote $x_i = X_i + h$ and $x_j = X_j + h$. Then, \eqref{eq:sym1} and \eqref{eq:sym2} imply
\[
	|u (x_i ) - u(x_j)| \leq \sum_{k \neq i,j} |w_k(x_i) - w_k(x_j)| + C \frac{|X_i - X_j|}{N |X_i - X_j|^2}.
\]

	 For all $k \neq i,j$ we use Lemma \ref{lem:semicircleCurve},
which provides curves $s_k \in C^1([0,1];\IR^3)$ from $x_i$ to $x_j$ such that 
	\[
		|s_k(t) - X_k| \geq \min\{|x_i-X_k|,|x_j-X_k|\},
	\]
	and
	\[
			|\dot{s_k}| \leq C|x_i-x_j|.
	\]
We deduce
\begin{equation}
	\label{eq:lemOsz1}
\begin{aligned}
	|w_k(x_i) - w_k(x_j)| &\leq 
	C \int_0^1 |\nabla w_k (s_k(t))||x_i-x_j| \dd t \\
	& \leq C  \frac{|X_i-X_j|}{N}\left( \frac{1}{|x_i-X_k|^2} + \frac{1}{|x_j-X_k|^2} \right).
\end{aligned}
\end{equation}
Thus, using Lemma \ref{lem:pointsOutsideParticles}, we conclude
\begin{align}
	|u(x) - u(y) | &\leq \sum_{k \neq i,j} |w_k(x_i) - w_k(x_j)| 
	+ C \frac{|X_i - X_j|}{N |X_i - X_j|^2} \\
	& \leq C \sum_{i \neq j,k} \frac{|X_i-X_j|}{N}\left( \frac{1}{|x_i-X_k|^2} + \frac{1}{|x_j-X_k|^2} \right) 
	+ C \frac{|X_i - X_j|}{N |X_i - X_j|^2}\\
	& \leq C \alpha |X_i-X_j|. \qedhere
\end{align}

\end{proof}

\begin{lemma}
	\label{lem:semicircleCurve}
	Let $x,y,z \in \IR^3$ be distinct. Then, there exists a curve $s \in C^1([0,1];\IR^3)$
	with $s(0) = x$, $s(1) = y$, 
	\[
		|s(t) - z| \geq \min\{|x-z|,|y-z|\},
	\]
	and
	\[
			|\dot{s}| \leq C|x-y|.
	\]
\end{lemma}

\begin{proof}
	Assume the points $x,y,z$ are not collinear. Let $E$ be the plane that contains $x$ and $y$ and
	is perpendicular to $z$. Then $\IR^3 \backslash E$ consists of two halfspaces.
	Of those two halfspaces, we denote by $H$ the halfspace that contains $z$.
	Then, we choose $s(t)$ to be a parametrization of the semicircle that is uniquely determined
	by the following three properties:
\begin{enumerate}[label=(\roman*)]
		\item It joins $x$ and $y$.
		
		\item It lies in the same plane as $z$. 
		
		\item The semicircle is disjoint to the halfspace H.
	\end{enumerate}
	If $x,y,z$ are collinear, we just take any semicircle joining $x$ and $y$.
	
	Now, it is easy to check that $s(t)$, if parametrized by constant speed, has all the desired properties.
\end{proof}

Using Proposition \ref{pro:LInftyEstimateSeries}, we could deduce from Lemma \ref{lem:OszillationsOfVelocity} 
\begin{equation}
	\label{eq:brutalOszillationOfVelocity}
	|v_\eps (X_i ) - v_\eps(X_j)| \leq C \alpha |X_i-X_j| + C\alpha( \alpha \phi_\eps + R_\eps),
\end{equation}
provided that the assumptions of both Proposition \ref{pro:LInftyEstimateSeries} and 
Lemma \ref{lem:OszillationsOfVelocity} are satisfied.

The particle volume $\phi_\eps$ on the right hand side of \eqref{eq:brutalOszillationOfVelocity} poses a problem,
since it could be much larger
than the minimal particle distance $d_{\eps,\min}$.
Thus, not even for small times $t$, does estimate \eqref{eq:brutalOszillationOfVelocity} imply any lower estimate 
for $d_{\eps,\min}(t)$ which is uniform in $\eps$.

In order to get rid of $\phi_\eps$ in \eqref{eq:brutalOszillationOfVelocity}, we will prove that the functions 
$u_k$ from Proposition \ref{pro:LInftyEstimateSeries} all satisfy
\[
	\sup_k |u_k (X_i) - u_k(X_j)| \leq C \alpha |X_i-X_j|.
\]

\begin{lemma}
	\label{lem:OszillationsOfVelocityApproximation}
	There is a constant $\delta > 0$ with the following property.
	Assume that
	\[
		\sup_j \frac{\phi}{N} \sum_{i \neq j} \frac {1}{|X_i - X_j|^3} \leq \delta,
	\] 
	and assume there is some constant $\alpha < \infty$ such that
	\[
		\sup_j\frac{1}{N}  \sum_{i \neq j} \frac {1}{|X_i - X_j|^2} \leq \alpha.
	\] 
	Then, the functions $u_k$
	defined in Proposition \ref{pro:LInftyEstimateSeries} satisfy for all particles 
	$i$ and $j$
	\begin{equation}
		\label{eq:OszillationsOfVelocityApproximation}
		|u_k (X_i) - u_k(X_j)| \leq  C \alpha |X_i-X_j|,
	\end{equation}
	In particular,
	\begin{equation}
	\label{eq:OszillationsOfRealVelocity}
		|v(X_i) - v(X_j)| \leq  C \alpha |X_i-X_j|.
	\end{equation}
\end{lemma}

\begin{proof}
	The assertion follows from the following estimate which we will prove by induction in $k$.
	\begin{equation}
		\label{eq:OszApprInduction}
		|u_k (X_i + h) - u_k(X_j+h)| 
		\leq C \delta^{-1} \alpha  \sum_{n = 0}^k (C \delta)^n |X_i-X_j|,
	\end{equation}
	for all particles $i,j$ and all $h \in \overline{B_R(0)}$.
	For $k = 0$, this is the second part of Lemma \ref{lem:OszillationsOfVelocity}.
	
	Let us denote $x_i = X_i + h$ and $x_j = X_j + h$.
	Using the definition of $u_k$ from Proposition \ref{pro:LInftyEstimateSeries}, we observe
	\begin{equation}
		u_{k+1}(x_i) = u_k(x_i) - \sum_l Q_l u_k(x_i) = (u_k)_i - \sum_{l \neq i} Q_l u_k(x_i).
	\end{equation}
	Here we used that by Lemma \ref{lem:projectionInsideParticle},
	$Q_i u_k = u_k - (u_k)_i$ in $\overline{B_i}$, where $(u_k)_i = \fint_{\partial B_i} u_k$.
	Therefore,
	\begin{equation}
	\begin{aligned}
		&|u_{k+1} (x_i) - u_{k+1}(x_j)|	\\ &\leq |(u_k)_i - (u_k)_j| 
		+ \sum_{l \neq i,j} |Q_l u_k (x_i) - Q_l u_k (x_j)|
		+ |Q_j u_k (x_i)| + |Q_i u_k (x_j)| \\
	\end{aligned}
		\label{eq:OszApprSplit}
	\end{equation}
	For the first term on the right hand side, we use the induction hypothesis.
	Regarding the second term, 
	for all $l \neq i,j$, we use Lemma \ref{lem:semicircleCurve},
which provides curves $s_l \in C^1([0,1];\IR^3)$ from $X_i$ to $X_j$ such that 
	\[
		|s(t) - X_l| \geq \min\{|X_l-x_i|,|X_l-x_j|\},
	\]
	and
	\[
			|\dot{s}| \leq C|x_i-x_j|.
	\]
Using in addition Estimates \eqref{eq:dipoleDecay1} and \eqref{eq:NablaUTilde} from Proposition \ref{pro:LInftyEstimateSeries}, we deduce
	\begin{align}
		|Q_l u_k (x_i) - Q_l u_k (x_j)| & \leq 
		C |x_i-x_j| \int_0^1 |\nabla Q_l u_k (s_l(t))| \dd t \\
		& \leq C |X_i-X_j| R^{3} \left( \frac{1}{{|x_i - X_l|}^3} + \frac{1}{{|x_j - X_l|}^3} \right)
		\|\nabla u_{k}\|_{L^\infty(B_l)} \\
		& \leq C \alpha  |X_i-X_j| R^{3} 
		\left( \frac{1}{{|X_i - X_l|}^3} + \frac{1}{{|X_j - X_l|}^3} \right) (C \delta)^{k}.
	\end{align} 
	Thus, for the second term on the right hand side of \eqref{eq:OszApprSplit}, we deduce
	\begin{align}
		\sum_{l \neq i,j} |Q_l u_k (x_i) - Q_l u_k (x_j)|
	&\leq C \alpha  |X_i-X_j| (C \delta)^{k+1}.
	\end{align}
	For the third term on the right hand side of \eqref{eq:OszApprSplit}
	we observe that Estimates \eqref{eq:dipoleDecay2} and \eqref{eq:NablaUTilde} from Proposition \ref{pro:LInftyEstimateSeries} yield
	\begin{align}
		|Q_j u_k (x_i)| 
		 \leq C \frac{R^{3}}{|X_i - X_j|^2} \|\nabla u_{k}\|_{L^\infty(B_l)} 
		 \leq C \alpha \frac{R^{3}}{|x_i - X_j|^2} (C \delta)^{k} \leq C \delta{-1} \alpha (C \delta)^{k+1}|X_i-X_j|,
	\end{align}
	where we used $R < |X_i - X_j|$.
	
	Since we get the same estimate for the fourth term, this finishes the proof of the
	induction step. Thus, estimate \eqref{eq:OszApprInduction} holds true for
	all $k \in \IN$ which implies \eqref{eq:OszillationsOfVelocityApproximation}.
	Since $u_k$ converges to $v$ in $L^\infty(\IR^3)$ by Proposition \ref{pro:LInftyEstimateSeries},
	 this also proves \eqref{eq:OszillationsOfRealVelocity}.
\end{proof}

\begin{proposition}
	\label{pro:particlesRemainSeparated}
	Assume for some time $T_0 \geq 0$, there exists $\eps_0 >0$ and $Y_0 < \infty$ such that 
	$Y_\eps(T_0) \leq Y_0$ for all $	\eps < \eps_0$.
	Then, there exists $\eps_1 >0$ and $\theta > 0$, which depends only on $Y_0$ and $c_0$ from 
	Assumption \ref{cond:particlesSeparated} such that 
	\[
		Y_\eps(T_0 + \theta) \leq 2 Y_\eps (0)  \qquad \text{for all} \quad \eps < \eps_1.
	\]
\end{proposition}

\begin{proof}
	By Lemma \ref{lem:brutalEstimatesSums}, 
	\begin{equation}
		\sup_j \frac{\phi_\eps}{N_\eps} \sum_{i \neq j} \frac {1}{|X_i - X_j|^3} 
		\leq C_\ast \phi_\eps Y_\eps^3 (\log (N_\eps) + \log(Y_\eps)).
	\end{equation}
	We choose $\eps_1 < \eps_0$ such that
	\[
		 C_\ast \phi_\eps (2 Y_0)^3 (\log (N_\eps) + \log(2 Y_0)) 
		 < \delta \qquad \text{for all} \quad \eps < \eps_1
	\]
	where $\delta$ is the constant from Lemma \ref{lem:OszillationsOfVelocityApproximation}.
	This is possible, since $\phi_\eps \log (N_\eps) \to 0$ as $\eps \to 0$ by Assumption \ref{cond:phiLogN}.
	Let 
	\[
		\theta_\eps := \sup\{ t \geq 0 \colon Y_\eps(T_0 + t) \leq 2 Y_0 \}	.
	\]	
	Since $Y_\eps$ is continuous in time, $\theta_\eps > 0$.
	Then, Lemma \ref{lem:OszillationsOfVelocityApproximation} and Lemma \ref{lem:brutalEstimatesSums} yield
	\[
		|v(X_i) - v(X_j)| \leq  C_\ast Y_0^3 |X_i-X_j| \qquad \text{for all} \quad \eps < \eps_1 
		\quad \text{and all} \quad t \leq T_0 + \theta_\eps.
	\]
	Since the particles are transported by $v$, this implies
	\[
		\dot{d}_{\eps,\min}(t) \geq C_\ast Y_0^3 d_{\eps,\min}(t) \qquad \text{for all} \quad \eps < \eps_1 
		\quad \text{and all} \quad t \leq T_0 + \theta_\eps.
	\]
	Hence,
	\[
		Y_\eps (T_0 + t) \leq Y_0 e^{C_\ast Y_0^3 t}   \qquad \text{for all} \quad \eps < \eps_1 
		\quad \text{and all} \quad t \leq \theta_\eps.
	\]
	By definition of $\theta_\eps$, this implies 
	\[
		\theta_\eps \geq \frac{\log 2}{C_\ast Y_0^3} =: \theta,
	\]
	which finishes the proof.
\end{proof}

\section{Approximations for the macroscopic fluid velocity}

In order to prove the convergence result Theorem \ref{th:main}, we need to relate the microscopic
fluid velocity $v_\eps$ to the macroscopic transport velocity $v_\ast$ in \eqref{eq:transportStokes}.
More precisely, we have to prove that inside the particles, $v_\eps$ is close to $v_\ast + \frac{2}{9}\xi_\ast^2 e$.
However, since $v_\ast = S(\rho e)$, it is convenient to compare $v_\ast$ to 
$\tilde{u}_\eps = S(\rho_\eps e)$, which has been introduced in Section
\ref{sec:VariationalProblem}. By Proposition \ref{pro:LInftyEstimateSeries}, we already know that we can replace 
$v_\eps$ by $u_\eps$. The following lemma provides an estimate between $\tilde{u}_\eps$ and $u_\eps$.

\begin{lemma}
	\label{lem:estimateUUTilde}
	Assume there is some constant $\alpha < \infty$ such that
	\[
		\sup_j\frac{1}{N}  \sum_{i \neq j} \frac {1}{|X_i - X_j|^2} \leq \alpha.
	\] 
	Let $\tilde{u}_\eps = S(\rho_\eps e)$, 
	$u_j = S \tilde{f}_j$, and $w_i = S f_j$, 
	where $\tilde{f}_j$ and  $f_j  $  are defined as in \eqref{eq:defTildeFi} and  \eqref{eq:defFi}.
	Then,
		\begin{equation}
		\label{eq:UUTildeWithoutSelfInteraction}
		\| (u_\eps - w_j) - (\tilde{u}_\eps - u_j) \|_{L^\infty(B_j)} \leq C \alpha R.
	\end{equation}
\end{lemma}

\begin{proof}
	We notice $u_i - w_i \in W_i^\perp$ (see \eqref{eq:characterizationWPerp}).
	In particular, $\supp \Delta(u_i - w_i) \subset B_i$ as a function in 
	$\dot{H}^{-1}_\sigma(\IR^3)$, and therefore,	 $u_i - w_i$ is the function of minimal norm in
	 $\Hdot_\sigma(\IR^3)$ that coincides with $u_i - w_i$ in $B_i$.
	Moreover, by Lemma \ref{lem:projectionInsideParticle},
	\[
		\fint_{\partial{B_i}} u_i(y) - w_i (y) \dd y = 0.
	\]
	Thus, by Lemma \ref{lem:extest},
	\begin{equation}
		\label{eq:UUTilde1}
		\| u_i - w_i \|_{\Hdot(\IR^3)} \leq \| \nabla (u_i - w_i) \|_{L^2(B_i)} 
		= \| \nabla u_i\|_{L^2(B_i)} \leq C R^{\frac{3}{2}} \| \nabla u_i\|_{L^\infty(B_i)} 
		\leq C \frac{R^{\frac{3}{2}}}{N R^2},
	\end{equation}
	where the last inequality follows using the fundamental solution of the Stokes equations.

	Together with Lemma \ref{lem:pCorrelationest} applied to $\Delta(u_i - w_i)$, estimate \eqref{eq:UUTilde1} yields
	for all $x  \in \IR^3 \backslash B_{2R}(X_i)$
	\begin{equation}
		\label{eq:UUTilde3}
		| u_i(x) - w_i(x)| 
		 \leq C \frac{R}{N |x - X_i|^2} 
	\end{equation}
	For $x \in B_j$, summing over all $i \neq j$ finishes the proof of \eqref{eq:UUTildeWithoutSelfInteraction}.
\end{proof}

\begin{lemma}
	\label{lem:StokesApprox}
	Let $w_{\delta,\eps} = S(\rho_\eps^{\delta}e) + \frac{2}{9} \xi_\eps^2 e$, 
	where $\rho_\eps^\delta$ is the averaged mass density of the particles according to Definition
	 \ref{def:CubeAverages}. 
	 Then,
	we have for all $\delta \geq d_{\eps,\min}(t)$
	\begin{equation}
		\label{eq:StokesApprox}
		\| u_\eps(t,\cdot) - w_{\delta,\eps}(t,\cdot) \|_{L^\infty(\cup_i B_i)} 
		\leq C_\ast \delta Y_\eps(t)^3.
	\end{equation}	
	Moreover,
	\begin{equation}
		\label{eq:wLipschitz}
		\|w_{\delta,\eps}\|_{L^\infty(0,t;W^{1,\infty}(\IR^3))} 
		\leq C_\ast Y_\eps(t)^3.
	\end{equation}		
	Furthermore, for all $\lambda_1, \lambda_2 \geq d_{\eps,\min}(t)$,
	\begin{equation}
		\label{eq:TwoStokesApprox}
		\| w_{\lambda_1,\eps}(t,\cdot) - w_{\lambda_2,\eps}(t,\cdot) \|_{L^\infty(\IR^3)} 
		\leq C_\ast Y_\eps(t)^3 \max\{\lambda_1,\lambda_2\} + C \xi_\eps^2.
	\end{equation}
\end{lemma} 

\begin{proof}
	Fix $\eps > 0$. Then, $u_\eps = \sum_i w_i$, where 
	$w_i = S f_i$ with $f_i$ as in \eqref{eq:defFi}. Since by \eqref{eq:2Over9} $w_i(x) = \frac{2}{9} \xi^2 g$ in $B_i$, we have for 
	all $x \in B_i$
	\[
		u_\eps(x) - (S(\rho_\eps^{\delta}e)(x) + \frac{2}{9} \xi_\eps^2 e) 
		= u_\eps(x) - w_i(x) - (S(\rho_\eps^{\delta}e)(x).
	\] 
	Hence, Lemma \ref{lem:estimateUUTilde} yields
	\[
		| u_\eps(x) - (S(\rho_\eps^{\delta}e)(x) + \frac{2}{9} \xi_\eps^2 e)|
		\leq  |\tilde{u}_\eps(x) - u_i(x) - (S(\rho_\eps^{\delta}e)(x)| + C \alpha R,
	\] 
	where $u_i := S f_i$ and 
	we recall from \eqref{eq:defFi}
	\[
		f_i = \frac{e}{\phi_\eps} \chi_{B_i}. 
	\]	
	We can write
	\[
		\tilde{u}_\eps - u_i = S(\rho_\eps e - f_i) 
	\]
	
	We recall that the solution operator $ S $ for the Stokes equations is a convolution operator, where the convolution kernel is the 
	Oseen tensor
\begin{equation}
	\label{eq:OseenTensor}
	\Phi(x) = \frac{1}{8\pi}\left( \frac{1}{|x|} + \frac{x \otimes x}{|x|^3} \right).
\end{equation}

We denote by $\Gamma_{\delta}$ the set of centers of the cubes from 
Definition \ref{def:CubeAverages}. We define $I_1 \subset \Gamma_{\delta}$ to contain the center of the cube
$Q_{\delta}^{x}$ as well as the centers of all cubes adjacent to $Q_{\delta}^{x}$.
Then $|I_1| = 27$. Let $I_2 \subset \Gamma_{\delta}$ be the centers of those remaining cubes which are not 
disjoint to the support of $\rho_\eps(t)$. 
We observe that for all $z \in \IR^3$
\[
	\int_{Q_{\delta}^z} \Phi(x-y) e \left(\rho_\eps(y) - \fint_{Q_{\delta}^\lambda} \rho_\eps(z') \dd z'\right) \dd y
	= \int_{Q_{\delta}^z} \left(\Phi(x-y) -  \fint_{Q_{\delta}^\lambda} \Phi(x-z') \dd z'\right)\rho_\eps(y) \dd y.
\]
Thus,
\begin{align}
	|S(\rho_\eps e - f_i)(x) -S(\rho_\eps^{\delta} e)(x)|
	&= \left|  \int_{\IR^3} \Phi(x-y) e \left(\rho_\eps(y) - \frac{\chi_{B_i}(y)}{\phi_\eps} - 
										\fint_{Q_{\delta}^y} \rho_\eps(z) \dd z\right) \dd y\right| \\
	&\leq \sum_{\lambda \in I_1} \int_{Q_{\delta}^\lambda} |\Phi(x-y)| 
	\left(\rho_\eps(y) - \frac{\chi_{B_i}(y)}{\phi_\eps} + \rho_\eps^{\delta}(y)\right) \dd y \\
	{} &+ \sum_{z \in I_2}  
	\int_{Q_{\delta}^z} \left|\Phi(x-y) -  \fint_{Q_{\delta}^z} \Phi(x-z') \dd z'\right|\rho_\eps(y) \dd y\\
	&=: A + B.
\end{align}
Recalling $\rho_\eps = \frac 1{\phi_\eps} \sum_i \chi_{B_i}$ leads to
\begin{equation}
	\label{eq:boundAveragedDensity}
	\rho_\eps^{\delta} (y) = \frac{R_\eps^3}{\phi_\eps \delta^3} |\{X_i \in Q_\delta^y\}| \leq
	 \frac{C}{N_\eps {d_{\eps,\min}(t)}^3}  \leq C_\ast Y_\eps(t)^3.
\end{equation}
Using this as well as $|\Phi(x)| \leq C/|x|$, we deduce
\begin{align}
	A &\leq C \sum_{\substack{j \in B_{C \delta}(x) \\ j \neq i}} \frac{1}{N_\eps |x - X_j|}
	  +\int_{B_{2{\delta}}(x)} \frac{1}{|x-y|} \dd y \\
	&\leq\left(1 + C\frac{1}{N_\eps d_{\eps,\min}(t)^3}\right) \int_{B_{C{\delta}}(x)} \frac{1}{|y-x|} \dd y \\
	& \leq C_\ast Y_\eps(t)^3 {\delta}^2.
\end{align}
From the explicit expression of $\Phi$ in \eqref{eq:OseenTensor}, it follows for all $z\in I_2$
and all $y \in Q_{\delta}^z$
\begin{equation}
	\label{eq:AveragePhi}
	\left|\Phi(x-y) -  \fint_{Q_{\delta}^z} \Phi(x-z') \dd z'\right| \leq C \frac{\delta}{|x-z|^2}.
\end{equation}
We define
	\[
		M(x) := |\{ X_i \in Q_{\delta}^x \}|
	\]
	Note that
	\[
		\| M \|_{L^1(\IR^3)} = \delta^3 N.
	\]
	Moreover, 
	\[
		\| M \|_{L^\infty(\IR^3)} \leq C  \left(\frac{\delta}{d_{\eps,\min}(t)}\right)^3  
								\leq C_\ast Y_\eps(t)^3 \delta^3 N 				
	\]
	Combining the $L^\infty$- and $L^1$-estimates of $M$ with \eqref{eq:AveragePhi} yields
	\[
		B \leq C \frac{\delta}{ N}  \sum_{z \in I_2}   \frac {M(z)}{|z - x|^2} 
		 \leq \frac{C \delta}{N \delta^3}  \int_{\IR^3} \frac {M(y)}{|y - X_j|} \dd y 
		 \leq C_\ast Y_\eps(t)^3 \delta. 
	\]
	Combining the error estimates for $A$ and $B$ proves \eqref{eq:StokesApprox}.
	
	The proof of \eqref{eq:TwoStokesApprox} is almost completely analogous. The only difference is that,
	due to the averaging, there in no problem with a particle that is close to the point where we estimate.
	Therefore, the estimate holds true in the whole of $\IR^3$.
	
	By using again the fundamental solution, estimate \eqref{eq:wLipschitz} is a direct consequence of the estimates \eqref{eq:boundAveragedDensity}
	and
	\[
		\| \rho_\eps^{\delta} \|_{L^1(\IR^3)} = C. \qedhere
	\]
\end{proof}

\section{Convergence to the macroscopic equation}

\subsection{Convergence for small times}

We first prove the main theorem \ref{th:main} up to times, for which $Y_\eps$ is unifomly bounded in $\eps$ for small $\eps$.

	We already know from Proposition \ref{pro:particlesRemainSeparated} that there exists such
	a time $T_0>0$.
	In Section \ref{sec:ConvergenceArbitraryTimes}, we will prove that $Y_\eps$ is actually uniformly bounded for small $\eps$ for every finite time interval.

\begin{theorem}
	\label{th:convergenceMacro}
	Assume conditions \ref{cond:particlesSeparated}-\ref{cond:screeningLength} are satisfied.
	Moreover, assume that for $T_0 > 0$ there exists an $\eps_0 > 0$ and $C_1 < \infty$ such that 
	\[
		Y_\eps(T_0) \leq C_1, 
		\qquad \text{for all} \quad \eps  < \eps_0.
	\]
	Let $\tilde{\delta}_\eps \to 0 $, such that $\tilde{\delta}_\eps = n_\eps \delta_\eps$ for some $n_\eps \in \IN^\ast$ with
	$n \to \infty $ as $\eps \to 0$. Then, if Assumption   \ref{ass:ConvergenceInitialData} is satisfied with some $\beta > 2$,
	\[
		\rho_\eps^{\tilde{\delta}_\eps} \to \rho  \quad \text{in} ~ L^\infty([0,T_0);X_\beta),
	\]
	where $\rho$ is the unique solution to problem \eqref{eq:transportStokes}.
\end{theorem}

	We do not prove smallness of $\rho_\eps^{\tilde{\delta}_\eps} - \rho$ directly.
	Instead, we introduce `intermediate' mass densities $\tau_\eps$ and $\sigma_\eps$.
	To this end, we denote again 
 \[
 	w_{\delta_\eps}:=S(\rho_\eps^{\delta_\eps}e) + \frac{2}{9} \xi_\eps^2 g,
 \]
 which it the approximation for the macroscopic fluid velocity studied in Section 7.
 We define $\tau_\eps$ to be the solution to
	\begin{align}
		\tau_\eps(0,\cdot) &= \rho_{\eps,0}, \\
		\partial_t \tau_\eps + w_{\delta_\eps} \cdot \nabla \sigma_\eps  &= 0,
	\end{align}
	and $\sigma_\eps$ to be the solution to
	\begin{align}
		\sigma_\eps(0,\cdot) &= \rho_{0,\eps}^{\tilde{\delta}_\eps}, \\
		\partial_t \sigma_\eps + 	w_{\delta_\eps} \cdot \nabla \sigma_\eps &= 0.
	\end{align}
	Then, the difference between $\rho_\eps$ and $\tau_\eps$ lies only in the transport velocity,
	and the difference between $\tau_\eps$ and $\sigma_\eps$ lies only in the initial datum.
	In Lemma \ref{lem:rhoTau}, we prove smallness of $\tau_\eps^{\tilde{\delta}_\eps} - \rho_\eps^{\tilde{\delta}_\eps}$, in Lemma \ref{lem:sigmaTau}, we prove smallness of
	$ \tau_\eps^{\tilde{\delta}_\eps} - \sigma_\eps$.
	Then, the proof of Theorem \ref{th:convergenceMacro} reduces to proving smallness of
	$\tau_\eps - \rho$.

\begin{notation}
	In this section,  any constant $\tilde{C}$ might depend on $c_0$, the fixed time $T_0$,  $C_1$ and $\|\nabla \rho_0\|_{X_\beta}$.
\end{notation}

\begin{lemma}
	\label{lem:rhoTau}
Under the assumptions of Theorem \ref{th:convergenceMacro},
\[
	\| \tau_\eps^{\tilde{\delta}_\eps} - \rho_\eps^{\tilde{\delta}_\eps} \|_{L^\infty([0,T_0);{X_\beta})} \to 0,
\]
where $\tau_\eps^{\tilde{\delta}_\eps}$ and $\rho_\eps^{\tilde{\delta}_\eps}$ are averages on cubes as in Definition
\ref{def:CubeAverages}.
\end{lemma}

\begin{proof}
	By Proposition \ref{pro:LInftyEstimateSeries} and Lemma \ref{lem:brutalEstimatesSums}, we have
	\[
		\| u_\eps - v_\eps\|_{L^\infty(\IR^3)} \leq C_\ast Y_\eps(t)^6(\phi_\eps + R_\eps).
	\]
	Combining this with Lemma \ref{lem:StokesApprox} yields
	\begin{equation}
		\label{eq:vW}
		\| w_{\delta_\eps} - v_\eps\|_{L^\infty(\IR^3)} 
		\leq C_\ast Y_\eps(t)^6(\phi_\eps + R_\eps + \delta_\eps)
		\leq \tilde{C} \delta_\eps
	\end{equation}

 We denote by $\psi_\eps$ and $\tilde{\psi}_\eps$ the flow of $v_\eps$ and 
$w_{\delta_\eps}$, respectively.
More precisely, $\psi_\eps: [0,T_0) \times [0,T_0) \times \IR^3 \to \IR^3 $ is the solution to
	\begin{equation}
	\begin{aligned}
		\partial_s \varphi(t,s,x) &= v(s,\varphi(t,s,x)), \\
		\varphi(t,t,x) &= x,
	\end{aligned}
	\end{equation}
	and analogously for $\tilde{\psi}_\eps$.
Let $x \in B_i(0)$. Then, $\psi_\eps(0,t,x) \in B_i(t)$ since $B_i$ is transported by $v_\eps$. Therefore, using \eqref{eq:vW} and Lemma \ref{lem:StokesApprox}, we deduce
\begin{align}
|\psi_\eps(0,t,x) -\tilde{\psi}_\eps(0,t,x)| 
		& \leq \int_0^t |v_\eps(s,\psi_\eps(0,s,x)) - w_{\delta_\eps}(s,\tilde{\psi}_\eps(0,s,y))| \dd s \\
		& \leq  \int_0^t |v_\eps(s,\psi_\eps(0,s,x)) - w_{\delta_\eps}(s,{\psi}_\eps(0,s,y))| \dd s \\
		 {} & + 	\int_0^t |w_{\delta_\eps}(s,\psi_\eps(0,s,x)) - w_{\delta_\eps}(s,\tilde{\psi}_\eps(0,s,y))| \dd s \\
		&\leq \tilde{C} \delta_\eps t + 
		C_\ast Y_\eps(t)^3 \int_0^t 	|\psi_\eps(0,s,x)) - \tilde{\psi}_\eps(0,s,x))| \dd s.		
\end{align}
Gronwall's inequality implies
\begin{equation}
	\label{eq:DifferencePsiPsiTilde}
	|\psi_\eps(0,t,x) -\tilde{\psi}_\eps(0,t,x)| \leq \tilde{C} \delta_\eps t e^{C_\ast Y_\eps(t)^3 t}
	\leq \tilde{C} \delta_\eps =: \gamma_\eps.
\end{equation}.

Consider a particle $i$. Then, \eqref{eq:DifferencePsiPsiTilde} implies that its mass transported by $w_{\delta_\eps}$
instead of $v_\eps$ lies in $B_{\tilde{C} \gamma_\eps}(X_i(t))$ at time $t$.
For $x \in \IR^3$, we define
\[
	q^x = \{ y \in Q_{\tilde{\delta}_\eps}^x \colon 
	\dist \{y, \partial Q_{\tilde{\delta}_\eps}^x\} > \tilde{C} \gamma_\eps \},
\]
and 
\[
	\bar{q}^x = \{ y \in \IR^3 \colon 
	\dist \{y, Q_{\tilde{\delta}_\eps}^x\} < \tilde{C} \gamma_\eps \}
\]
Then, \eqref{eq:DifferencePsiPsiTilde} yields
\[
	\frac{1}{|Q_{\tilde{\delta}_\eps}^x|} \int_{q^x} \rho_\eps(t,y) \dd y \leq 
	\fint_{Q_{\tilde{\delta}_\eps}^x} \tau_\eps(t,y) \dd y = \tau^{\tilde{\delta}_\eps}_\eps(t,x) 
	\leq \frac{1}{|Q_{\tilde{\delta}_\eps}^x|} \int_{\bar{q}^x} \rho_\eps(t,y) \dd y .
\]
Since we also have 
\[
	\frac{1}{|Q_{\tilde{\delta}_\eps}^x|} \int_{q^x} \rho_\eps(t,y) \dd y \leq 
	\fint_{Q_{\tilde{\delta}_\eps}^x} \rho_\eps(t,y) \dd y = \rho^{\tilde{\delta}_\eps}_\eps(t,x) 
	\leq \frac{1}{|Q_{\tilde{\delta}_\eps}^x|} \int_{\bar{q}^x} \rho_\eps(t,y) \dd y,
\]
it suffices to prove smallness of 
\begin{equation}
	\label{eq:sandwich}
	(1+|x|^\beta ) \left (\frac{1}{|Q_{\tilde{\delta}_\eps}^x|} \int_{q^x} \rho_\eps(t,y) \dd y
	- \frac{1}{|Q_{\tilde{\delta}_\eps}^x|} \int_{\bar{q}^x} \rho_\eps(t,y) \dd y \right)
	=  \frac{1+|x|^\beta }{|Q_{\tilde{\delta}_\eps}^x|} \int_{q^x \backslash \bar{q}^x} \rho_\eps(t,y) \dd y.
\end{equation}
Fix a particle $i$ such that $X_i(t) \in q^x \backslash \bar{q}^x$ and consider $B_{\delta_\eps}(X_i(t))$.
Then, by definition of $Y_\eps(T_0)$, we know that $X_j(t) \in B_{\delta_\eps}(X_i(t))$
implies $X_j(0) \in B_{Y(T_0)\delta_\eps}(X_i(0))$.
Thus,
\[
	\int_{B_{\delta_\eps}(X_i(t))} \rho_\eps (t,y) \dd y 
	\leq \int_{B_{Y(T_0) \delta_\eps}(X_i(0))} \rho_{\eps,0} (y)
\]
Let $I$ denote the set of centers $z$ of cubes $Q^z_{\delta_\eps}$ from Definition \ref{def:CubeAverages} 
with $Q^z_{\delta_\eps} \cap B_{Y_\eps(T_0) \delta_\eps}(X_i(0)) \neq \emptyset$.
Then,
\begin{align}
	\int_{B_{\delta_\eps}(X_i(t))} \rho_\eps (t,y) \dd y 
	&\leq \sum_{z \in I} \int_{Q^z_{\delta_\eps}}  \rho_{\eps,0} (y) \dd y \\
	&\leq \sum_{z \in I} \delta_\eps^3 \frac{1}{1+|z|^\beta} \|\rho_{\eps,0}\|_{{X_\beta}}.
\end{align}
Using the bound on $Y(T_0)$, we have $|I| \leq \tilde{C}$.
Furthermore, from Assumption \ref{ass:ConvergenceInitialData}, we know that $\rho_{\eps,0}$
is uniformly bounded in ${X_\beta}$.
Moreover, we observe that by \eqref{eq:vW} and Lemma \ref{lem:StokesApprox} $v_\eps$ is uniformly bounded in $L^\infty(\IR^3)$. 
Thus, since  $ X_i(t) \in q^x \backslash \bar{q}^x$ at time $t \leq T_0$, we have
$X_i (0) \in B_{\tilde{C}}(x)$ and also $z \in B_{\tilde{C}}(x)$ for all $z \in I$.
Therefore, 
\begin{equation}
	\label{eq:coverStart}
	\int_{B_{\delta_\eps}(X_i(t))} \rho_\eps (t,y) \dd y 
	\leq \tilde{C} \delta_\eps^3 \frac{1}{1+|x|^\beta}.
\end{equation}

Finally, we note that the number  $M$ of balls $B_{\delta_\eps}(X_i(t))$ with $X_i(t) \in q^x \backslash \bar{q}^x$ that are needed to cover all the particles $B_i(t)$ in 
$q^x \backslash \bar{q}^x$ is bounded by 
\[
	M \leq C\frac{|q^x \backslash \bar{q}^x|}{\delta_\eps^3} 
	\leq C \frac{\gamma_\eps \tilde{\delta}_\eps^2}{\delta_\eps^3} 
	\leq C \frac{\delta_\eps \tilde{\delta}_\eps^2}{\delta_\eps^3} 
	\leq C \frac{ \tilde{\delta}_\eps^2}{\delta_\eps^2}.
\]
Combining this with \eqref{eq:coverStart} yields
\[
	(1+|x|^\beta ) \frac{1}{|Q_{\tilde{\delta}_\eps}^x|} \int_{q^x \backslash \bar{q}^x} \rho_\eps(t,y) \dd y \leq M \tilde{C} \frac{\delta_\eps^3}{|Q_{\tilde{\delta}_\eps}^x|} 
	\leq \tilde{C} \frac{\delta_\eps}{ \tilde{\delta_\eps}} \to 0.
\]
This finishes the proof.
\end{proof}

\begin{lemma}
\label{lem:sigmaTau}
Under the assumptions of Theorem \ref{th:convergenceMacro}
	\[
		\| \tau_\eps^{\tilde{\delta}_\eps} - \sigma_\eps \|_{L^\infty([0,T];X_\beta)} \to 0.
	\]
\end{lemma}

\begin{proof}
The difference between $\tau_\eps^{\tilde{\delta}_\eps} $ and $\sigma_\eps$ 
is only the order of making the time evolution and taking averages.
More precisely, $\sigma_\eps$ is the function we get from transporting the averaged initial datum,
whereas $\tau_\eps^{\tilde{\delta}_\eps}$ is the average of the transported initial datum.

Denoting by $\psi_\eps$ the flow of $w_{\delta_\eps}$, we get
\[
	\sigma_\eps (t,x) = \fint_{Q_{\tilde{\delta}_\eps}^{\psi_\eps(t,0,x)}} \rho_{\eps,0}(y) \dd y.
\]
On the other hand, we have
\[
	\tau_\eps^{\tilde{\delta}_\eps}(t,x) = \fint_{Q_{\tilde{\delta}_\eps}^x} \rho_{\eps,0}(\psi(t,0,y)) \dd y 
					 = \fint_{\psi_\eps\left(t,0,{Q_{\tilde{\delta}_\eps}^x}\right)} \rho_{\eps,0}(y) \dd y,
\]
where we used that
$\det( D \psi_\eps) = 1$ which follows from the fact that $w_{\delta_\eps}$ is divergence free. 

We estimate using Lemma \ref{lem:StokesApprox}
\begin{align}
|\psi_\eps(0,t,x) -\psi_\eps(0,t,y)| 
		& \leq |x-y| + \int_0^t |w_{\delta_\eps}(s,\psi_\eps(0,s,x)) - w_{\delta_\eps}(s,\psi_\eps(0,s,y))| \dd s \\
		& \leq  |x-y| + C_\ast Y_\eps(t)^3 \int_0^t |\psi_\eps(0,s,x) - \psi_\eps(0,s,x)| \dd s.
\end{align}
Gronwall's inequality implies
\begin{equation}
	\label{eq:LipschitzPsi}
	|\psi_\eps(0,t,x) -\psi_\eps(0,t,y)| \leq | x-y| e^{ C_\ast Y_\eps(T_0)^3 T_0} \leq \tilde{C}|x-y|.
\end{equation}
By an analogous argument, we also get the lower bound
\begin{equation}
	\label{eq:LipschitzPsiLower}
	|\psi_\eps(0,t,x) -\psi_\eps(0,t,y)| \geq | x-y| e^{ -C_\ast Y_\eps(T_0)^3 T_0} 
	\geq \frac{1}{\tilde{C}}|x-y|.
\end{equation}

Consider a point $ y \in Q_{\tilde{\delta}_\eps}^{x}$ at time $t$. We want to find $\gamma_\eps$ such that
$\dist\{y,\partial Q_{\tilde{\delta}_\eps}^{x}\} > \gamma_\eps$ implies
\begin{equation}
	\label{eq:smallCubeToLargeCube}
	\psi_\eps(0,t,Q_{\delta_\eps}^{\psi_\eps(t,0,y)}) \subset Q_{\tilde{\delta}_\eps}^{x}.
\end{equation}
Estimate \eqref{eq:LipschitzPsi} implies that this is true with 
\begin{equation}
	\label{eq:gamma}
	\gamma_\eps =  \tilde{C} \delta_\eps \ll \tilde{\delta}_\eps,
\end{equation}
for all $t \leq T_0$.
Let 
\[
	q_\eps(x) = \{  y \in Q_{\tilde{\delta}_\eps}^{x} \colon \dist\{y,\partial Q_{\tilde{\delta}_\eps}^{x}\} > \gamma_\eps \},
\]
and
\[
	\bar{q}_\eps(t,x) = \bigcup_{y \in q_\eps(t,x)} Q_{\delta_\eps}^{\psi_\eps(t,0,y)}.
\]
Then, by \eqref{eq:smallCubeToLargeCube},
 \begin{equation}
 	\label{eq:qSubsets}
 	 q_\eps(x) \subset \psi_\eps(0,t,\bar{q}_\eps(t,x)) \subset Q_{\tilde{\delta}_\eps}^{x}.
 \end{equation}
Therefore,
\begin{align}
	(1+|x|^\beta) |\sigma_\eps (t,x) - \tau_\eps^{\tilde{\delta}_\eps}(t,x)| 
	&= (1+|x|^\beta) \left| \fint_{Q_{\tilde{\delta}_\eps}^{\psi_\eps(t,0,x)}} \rho_{\eps,0}(y) \dd y -
		\fint_{Q_{\tilde{\delta}_\eps}^x} \rho_{\eps,0}(\psi_\eps(t,0,y)) \dd y \right| \\
	 &\leq 
	 (1+|x|^\beta) \left| \fint_{Q_{\tilde{\delta}_\eps}^{\psi_\eps(t,0,x)}} \rho_{\eps,0}(y) \dd y 
	 - \rho_0(\psi_\eps(t,0,x)) \right| \\
	{} & + (1+|x|^\beta) \left| \frac{1}{|Q_{\tilde{\delta}_\eps}^x|}  \int_{\bar{q}_\eps(t,x)} \rho_{\eps,0}(y) \dd y
	- \rho_0(\psi_\eps(t,0,x)) \right| \\
	 {} & +  (1+|x|^\beta) \frac{1}{|Q_{\tilde{\delta}_\eps}^x|} 
	 \int_{Q_{\tilde{\delta}_\eps}^x \backslash q_\eps(t,x)} \rho_{\eps,0}(\psi_\eps(t,0,y) \dd y \\
	 &=: A_1 + A_2 + A_3.
\end{align}

By Lemma \ref{lem:StokesApprox} and the assumption on $Y(T_0)$, $w_{\delta_\eps}$ is uniformly bounded in $L^\infty(\IR^3)$. Thus, $|\psi_\eps(t,0,y)| \geq |y| - \tilde{C}$ and
\[
	\frac{1}{1+|y|^\beta} \leq \frac{\tilde{C}}{1+|x|^\beta}, 
	\qquad \text{for all} \quad y \in Q_{\tilde{\delta}_\eps}^{\psi_\eps(t,0,x)}.
\]

We estimate $A_1$ using the convergence $ \rho_{0,\eps}^{\delta_\eps} \to \rho_0$ and
boundedness of  $\|\nabla \rho_0\|_{X_\beta}$.
\[
	A_1 \leq \tilde{C} \| \rho^{\delta_\eps}_{\eps,0} - \rho_0 \|_{X_\beta}
	+ \tilde{\delta}_\eps \| \nabla \rho_0 \|_{X_\beta} \to 0.
\]

In order to estimate $A_3$, we proceed as in the estimate of the term 
in \eqref{eq:sandwich} from Lemma \ref{lem:rhoTau}. We have to control the number of deformed particles transported by $w_{\delta_\eps}$ in  
$Q_{\tilde{\delta}_\eps}^x \backslash q_\eps(t,x)$ at time $t$. 
To this end, we define the trajectories of the particles transported by $w_{\delta_\eps}$
\[
	\tilde{X}_i(t) := \psi_\eps(0,t,X_i(0))
\]
and
\[
	\tilde{B}_i(t) := \psi_\eps(0,t,B_i(0)).
\]
Then,  estimate \eqref{eq:LipschitzPsiLower}
implies for all $i \neq j$
\[
	|\tilde{X}_i(t) - \tilde{X}_j(t)| \geq \frac{\tilde{C}} {|{X}_i(0) - {X}_j(0)|}.
\]
and
\[
	\op{diam} \tilde{B}_i(t) \leq \tilde{C} R_\eps
\]
Therefore, $A_3$ tends to zero by the same argument as we have proved smallness of \eqref{eq:sandwich}.

For $A_2$, let $(x_i)_{i=1}^n$ denote the centers of the disjoint cubes that $\bar{q}_\eps(t,x)$
consists of. Note that \eqref{eq:qSubsets} implies $|q_\eps(x)| \leq |\bar{q}_\eps(t,x)|$
due to conservation of volume.
Using also \eqref{eq:LipschitzPsi}, we deduce
\begin{align}
	A_2 &\leq (1+|x|^\beta) \frac{|Q_{{\delta}_\eps}^{x}|}{|Q_{\tilde{\delta}_\eps}^x|} \sum_{i=1}^n
	 \left| \fint_{Q_{{\delta}_\eps}^{x_i}} \rho_{\eps,0}(y) \dd y - \rho_0(\psi_\eps(t,0,x)) \right| 
	 + \left(1 - \frac{|\bar{q}_\eps(t,x)|}{|Q_{\tilde{\delta}_\eps}^x|}\right) \rho_0(\psi_\eps(t,0,x)) \\
	 & \leq \tilde{C} \| \rho^{\delta_\eps}_{\eps,0} - \rho_0 \|_{{X_\beta}} 
	+ \tilde{C} \tilde{\delta}_\eps \| \nabla \rho_0 \|_{{X_\beta}} + 
	\tilde{C} \frac{\left|Q_{\tilde{\delta}_\eps}^x \backslash q_\eps(x)\right|}{|Q_{\tilde{\delta}_\eps}^x|}
	 \|\rho_0\|_{{X_\beta}} \\
	& \leq \tilde{C} \| \rho^{\delta_\eps}_{\eps,0} - \rho_0 \|_{{X_\beta}} 
	+ \tilde{C} \tilde{\delta}_\eps \| \nabla \rho_0 \|_{{X_\beta}} + 
	\tilde{C} \frac{\gamma_\eps}{\tilde{\delta}_\eps} \|\rho_0\|_{L^\infty(\IR^3)}.
\end{align}
By equation \eqref{eq:gamma}, this tends to $0$
as $\eps \to 0$.
\end{proof}

\begin{lemma}
\label{lem:SContinuousXToLipschitz}
	For all $\beta > 2$ and all $h \in {X_\beta}$, 
	\[
		\| S h \|_{W^{1,\infty}(\IR^3)} \leq C \|h \|_{X_\beta}.
	\]
\end{lemma}

\begin{proof}
We recall that the solution operator $S$ can be represented by the convolution with the 
Oseen tensor
\[
	\Phi(x) = \frac{1}{8\pi}\left( \frac{1}{|x|} + \frac{x \otimes x}{|x|^3} \right).
\]
Hence,  by definition of ${X_\beta}$,
\begin{align}
	|(S h)(x)| &\leq C \|h \|_{X_\beta}  \int_{\IR^3} \frac{1}{|x-y|} \frac{1}{|1 + |y|^\beta} \dd y \\
	&\leq C \|h \|_{X_\beta}\int_{B_{\frac{|x|}{2}}(x)} \frac{1}{|x-y|} \frac{1}{|1 + |x|^\beta}
	 + C \|h \|_{X_\beta}\int_{\IR^3 \backslash B_{\frac{|x|}{2}}(x)} \frac{1}{|y|} \frac{1}{|1 + |y|^\beta} \\
	 & \leq C \|h \|_{X_\beta}\frac{|x|^2}{1+|x|^\beta} + C \|h \|_{X_\beta},
\end{align}
since $\frac{1}{|y|} \frac{1}{|1 + |y|^\beta} \in L^1(\IR^3)$ as $1+ \beta > 3$.

The estimate for $\nabla (S h)(x)$ works analogously.
\end{proof}

\begin{proof}[Proof of Theorem \ref{th:convergenceMacro}]
	We define $\psi_\eps$ to be the flow of 
	$w_{\delta_\eps}=S(\rho_\eps^{\delta_\eps}e) + \frac{2}{9} \xi_\eps^2 e$. Moreover, we write
	$w = S(\rho e) + \frac{2}{9} \xi^2_\ast e $ and denote by $\tilde{\psi}$ the flow of $w$.
	
	We recall from Lemma \ref{lem:StokesApprox} that $w_{\delta_\eps}$
	is uniformly bounded  in $L^\infty((0,T_0) \times \IR^3)$.
	Moreover 
	\begin{equation}
		\label{eq:EstimateNablaV}
		\|w\|_{W^{1,\infty}((0,T_0)\times\IR^3)} \leq \tilde{C}
	\end{equation}
	 This
	follows from boundedness of $\rho$ in $L^\infty(0,T_0;{X_\beta})$, which is stated in Theorem \ref{th:existenceTransportStokes}, and Lemma \ref{lem:SContinuousXToLipschitz}.
	From the $L^\infty$-bounds on $w$ and $w_{\delta_\eps}$, we deduce for all $x \in \IR^3$
	\[
			\frac{1}{1+|{\psi_\eps(t,0,x)}|^\beta} \leq \frac{\tilde{C}}{1+|x|^\beta}, 
	\]
	and the same inequality with $\tilde{\psi}$ replacing $\psi$.

Let $\sigma_\eps$ be the function from Lemma \ref{lem:sigmaTau}.
Then,
	\begin{equation}
		\label{eq:Gronwall1}
	\begin{aligned}
		| \rho(t,x) - \sigma_\eps(t,x)| &= |\rho_0(\tilde{\psi}(t,0,x)) - \rho_{0,\eps}^{\tilde{\delta}_\eps}(\psi_\eps(t,0,x))| \\
		& \leq |\rho_0(\tilde{\psi}(t,0,x)) -\rho_0(\psi_\eps(t,0,x))| + | \rho_0(\psi_\eps(t,0,x)) - \rho_{0,\eps}^{\tilde{\delta}_\eps}(\psi_\eps(t,0,x))| \\
		&\leq \frac{1}{1+|x|^\beta} \left( \|\nabla \rho_0\|_{{X_\beta}} |\tilde{\psi}(t,0,x) -\psi_\eps(t,0,x)|  + \|\rho_0 - \rho_{0,\eps}^{\tilde{\delta}_\eps}\|_{{X_\beta}} \right).
	\end{aligned}
	\end{equation}

	Concerning the first term on the right hand side, we have
	\begin{equation}
	\label{eq:Gronwall2}
	\begin{aligned}
		|\tilde{\psi}(t,0,x) - \psi_\eps(t,0,x)| & \leq  \int_0^t |w(s,\tilde{\psi}(t,s,x)) - 
		w_{\delta_\eps}(s,\psi_\eps(t,s,x))| \dd s \\
		& \leq \int_0^t |w(s,\tilde{\psi}(t,s,x)) - w(s,\psi_\eps(t,s,x))| \dd s \\
		 {} &+ \int_0^t | w(s,\psi_\eps(t,s,x)) - w_{\delta_\eps}(s,\psi_\eps(t,s,x))| \dd s \\
		& \leq  \|\nabla w\|_{L^\infty} \int_0^t |\tilde{\psi}(t,s,x)) - \psi_\eps(t,s,x)| \dd s\\
		{} & + \int_0^t \| w(s,\cdot) -w_{\delta_\eps}(s,\cdot)\|_{L^\infty} \dd s.
	\end{aligned}
	\end{equation}
	Gronwall yields
	\begin{equation}
	\label{eq:Gronwall3}
	\begin{aligned}
		\|\tilde{\psi}(t,0,\cdot) - \psi_\eps(t,0,\cdot)\|_{L^\infty} 
		& \leq \int_0^t  \| w(s,\cdot) - w_{\delta_\eps}(s,\cdot)\|_{L^\infty} \dd s \\
		{} &+ \|\nabla w\|_{L^\infty} \int_0^t \int_0^s \| w(\tau,\cdot) - w_{\delta_\eps}(\tau,\cdot)\|_{L^\infty} \dd \tau e^{(t-s) \|\nabla w\|_{L^\infty}} \dd s \\
		& \leq  \left(t e^{t\|\nabla w\|_{L^\infty} } + 1\right) 
		\int_0^t  \| w(s,\cdot) - w_{\delta_\eps}(s,\cdot)\|_{L^\infty} \dd s.
	\end{aligned}
	\end{equation}
Combining estimates \eqref{eq:EstimateNablaV}, \eqref{eq:Gronwall1},  and \eqref{eq:Gronwall3}, we deduce
	for $t < T_0$
	\begin{equation} 
	\label{eq:GronwallStart}
	\begin{aligned}
		 \| \rho(t,\cdot) - \sigma_\eps(t,\cdot) \|_{{X_\beta}} & \leq  \|\rho_0 - \rho_{0,\eps}^{\tilde{\delta}_\eps}\|_{{X_\beta}} 
		 + \tilde{C} \|\nabla \rho_0\|_{{X_\beta}}
		 \int_0^t  \| w(s,\cdot) - w_{\delta_\eps}(s,\cdot)\|_{L^\infty} \dd s.
	\end{aligned}
	\end{equation}

	Lemma \ref{lem:SContinuousXToLipschitz} and Lemma \ref{lem:StokesApprox} yield
	\begin{equation}
	\begin{aligned}
		 \| w(s,\cdot) - w_{\delta_\eps}(s,\cdot)\|_{L^\infty} 
		 &\leq  \| w_{\tilde{\delta}_\eps,\eps}(s,\cdot)- w_{\delta_\eps}(s,\cdot)\|_{L^\infty} 
		  +  \| w(s,\cdot) -  w_{\tilde{\delta}_\eps,\eps}(s,\cdot)\|_{L^\infty}  \\
		 &\leq  C_\ast Y_\eps(t)^3  \tilde{\delta}_\eps  
		 +  \| S\left(\rho(s,\cdot) - \rho_\eps^{\tilde{\delta}_\eps}(s,\cdot)\right) \|_{L^\infty}\\
		 &\leq  \tilde{C}  \tilde{\delta}_\eps  
		 +  C_\ast \| \rho(s,\cdot) - \rho_\eps^{\tilde{\delta}_\eps}(s,\cdot)\|_{{X_\beta}} \\
		 & \leq \tilde{C}  \tilde{\delta}_\eps  
		 +  C_\ast \| \sigma_\eps(s,\cdot) - \rho_\eps^{\tilde{\delta}_\eps}(s,\cdot)\|_{{X_\beta}} 
		 		 + \tilde{C} \| \rho(s,\cdot) - \sigma_\eps(s,\cdot)\|_{{X_\beta}}\\
		 & =: \theta_1 +  \tilde{C} \| \rho(s,\cdot) - \sigma_\eps(s,\cdot)\|_{{X_\beta}}.
	\end{aligned}
		\label{eq:estimateDifferentVelocities}
	\end{equation}
	Note that $\theta_1 \to 0$ as $\eps \to 0$ by Lemma \ref{lem:rhoTau} and Lemma \ref{lem:sigmaTau}.
	Using estimate \eqref{eq:estimateDifferentVelocities} in
	 \eqref{eq:GronwallStart}, we deduce
	\begin{align}
		 \| \rho(t,\cdot) - \sigma_\eps(t,\cdot) \|_{{X_\beta}} & \leq  \|\rho_0 - \rho_{0,\eps}^{\tilde{\delta}_\eps}\|_{{X_\beta}} 
		  + \tilde{C} 
		\left(\theta_1 T_0 
		+ \int_0^t  \| \rho(s,\cdot) - \sigma_\eps(s,\cdot) \|_{{X_\beta}}  \dd s \right) .
	\end{align}
	We apply Gronwall once more to conclude
	\[
		 \| \rho(t,\cdot) - \sigma_\eps(t,\cdot) \|_{{X_\beta}} 
		 \leq (\tilde{C} \theta_1 + \|\rho_0 - \rho_{0,\eps}^{\tilde{\delta}_\eps}\|_{{X_\beta}} )
		  e^{t \tilde{C}},
	\]	
	which converges to zero, uniformly for $t \leq T_0$.
	Combining this estimate with Lemma \ref{lem:rhoTau} and Lemma \ref{lem:sigmaTau}
	finishes the proof.
\end{proof}

\subsection{Extension of the convergence to arbitrary times}
\label{sec:ConvergenceArbitraryTimes}

Using the convergence result, Theorem \ref{th:convergenceMacro}, we are able to prove a posteriori that the constant $C_1$ in the
assumption of Theorem \ref{th:convergenceMacro} does not blow up in finite time.
This will finally enable us to prove the main result, Theorem \ref{th:main}.

\begin{lemma}
	\label{lem:aPosterioriEstimateY}
	There exists a constant $C_2 $ which depends only on 
	$c_0$ from Assumtion \ref{cond:particlesSeparated} with the following property.
	Assume the assumptions of Theorem \ref{th:convergenceMacro} are satisfied for some time $T_0 >0$.
	Then, there exists  $\eps_1 > 0$ such that
		\begin{equation}
		Y_\eps(T_0)  \leq e^{ C_2 T_0}                                                                                                                                                                                                                                                                                                                                                      
		\qquad \text{for all} \quad \eps  < \eps_1, \quad \text{and} \quad t \leq T_0.
	\end{equation}
\end{lemma}

\begin{proof}
	The main issue is to control the term
	\[
		\sup_j \frac{1}{N_\eps} \sum_{i \neq j} \frac 1{|X_i - X_j|^k}
	\] 
	for $k = 2,3$, where the particle positions depend on $\eps$ and time.
	
	\emph{Claim.} There exists $\eps_1>0$ such that for all $\eps < \eps_1$,
	\[
		\sup_j\frac{1}{N_\eps}  \sum_{i \neq j} \frac {1}{|X_i - X_j|^2} 
		 \leq C_\ast(1+ \tilde{\delta}_\eps Y_\eps^3),
	\] 
	for some constant $C_\ast$, which only depends on $c_0$.
	
	Let $I$ be the set of the centers of the
	 cubes with side length $\tilde{\delta}_\eps$ from Definition \ref{def:CubeAverages}.
	At a fixed time $t < T_0$, we fix a particle $X_j$ and define $I_1$ to consist of the center of the
	 cube containing $X_j$,
	and the centers of the cubes that are adjacent to that cube. Furthermore, we denote $I_2 = I \backslash I_1$.
	Then, we estimate
	\begin{align*}
		\frac{1}{N_\eps} \sum_{i \neq j} \frac 1{|X_i - X_j|^2} &
		\leq \frac{1}{N_\eps}  \sum_{y \in I_1} 
		\sum_{X_i \in Q_{\tilde{\delta}_\eps}^y} \frac {1}{|X_i - X_j|^2} 
		+ \frac{1}{N_\eps}  \sum_{y \in I_2} 
		\sum_{X_i \in Q_{\tilde{\delta}_\eps}^y} \frac {1}{|X_i - X_j|^2} \\
		& =: A_1 + A_2.
	\end{align*}
	The first term, $A_1$, we estimate brutally,
	\[
		A_1 \leq \frac{C}{(d_{\eps, \min}(t))^3 N_\eps} 
		\int_{B_{C \tilde{\delta}_\eps}(X_j)} \frac{1}{|y-X_j|^2} \dd y 
		\leq C_\ast \tilde{\delta}_\eps^2 Y_\eps(t)^3 .
	\] 
	In order to estimate the second term, $A_2$, we define
	\[
		M(x) := |\{ X_i \in Q_{\tilde{\delta}_\eps}^x \}|.
	\]
	Note that
	\[
		\| M \|_{L^1(\IR^3)} = \tilde{\delta}_\eps^3 N_\eps.
	\]
	Moreover, 
	\[
		\rho_\eps^{\tilde{\delta}_\eps} = \frac{M}{N_\eps \tilde{\delta}_\eps^3}. 
	\]
	Thus, Theorem \ref{th:convergenceMacro} implies that we can choose $\eps_0$ small enough such that
	for $\eps < \eps_0$
	\[
		\| M \|_{L^\infty(\IR^3)} \leq 2 \tilde{\delta}_\eps^3 N_\eps \| \rho(t) \|_{L^\infty(\IR^3)} 
								= 2 \tilde{\delta}_\eps^3 N_\eps \| \rho(0) \|_{L^\infty(\IR^3)} 
									\leq \tilde{\delta}_\eps^3 N_\eps C_\ast,
	\]
	where we used that the $L^\infty$-norm of $\rho$ is conserved in time.
	Combining the $L^\infty$- and $L^1$-estimates of $M$ yields
	\[
		A_2 \leq \frac{C}{N_\eps}  \sum_{y \in I_2}   \frac {M(y)}{|y - X_j|^2} 
		 \leq \frac{C}{N_\eps \tilde{\delta}_\eps^3}  \int_{\IR^3} \frac {M(y)}{|y - X_j|^2} \dd y \leq C_\ast.
	\]
	Combining the estimates for $A_1$ and $A_2$ proves the claim.	 
	 
	Recall from Lemma \ref{lem:brutalEstimatesSums}
	\begin{align}
	\sup_j \frac{\phi_\eps}{N_\eps} \sum_{i \neq j} \frac {1}{|X_i - X_j|^3} 
		\leq C_\ast \phi_\eps Y_\eps^3 (\log (N_\eps) + \log(Y_eps)).
	\end{align}
	and this converges to zero for any fixed time $t < T_0$ due to 
	Assumption \ref{cond:phiLogN} since $Y_\eps(t)$ is bounded by assumption.
	
	Thus, Lemma \ref{lem:OszillationsOfVelocityApproximation} yields for all particles $i$ and $j$
	\[
		|v_\eps(t,X_i) - v_\eps(t,X_j)| \leq C_\ast(1 +  \tilde{\delta}_\eps Y_\eps(t)^3) |X_i - X_j| 	
	\] 
	 for all $ t \leq T_0 $ and all  $ \eps < \eps_0$ for some $\eps_0$ small enough.
	
	We estimate
	\[
		|X_i(t) - X_j(t)| \geq |X_i(0) - X_j(0)|  -  \int_0^t |v_\eps(s,X_i(s)) - v_\eps(s,X_j(s))| \dd s.
	\]
	Thus, by Gronwall
	\[
		|X_i(t) - X_j(t)| \geq |X_i(0) - X_j(0)| \exp \left( 
				-\int_0^t \frac{|v_\eps(s,X_i(s)) - v_\eps(s,X_j(s))|}{|X_i(s) - X_j(s)|} \dd s \right).
	\]
	Therefore, for all $\eps < \eps_0$,
	\begin{align}
		Y_\eps(t) &\leq  \exp \left( \int_0^t C_\ast(1 +  \tilde{\delta}_\eps Y_\eps(s)^3) \dd s \right) \\
		&\leq \exp \left( C_\ast t + C_\ast \tilde{\delta}_\eps \int_0^t Y_\eps(s)^3)  \dd s \right).
	\end{align}
	This means that $T_\eps$ might blow up in finite time for a fixed $\eps$ but 
	we can derive a lower bound for the blow up time that
	tends to infinity as $\delta_\eps \to 0$. To see this, we define for any $Y_0 > 0$
	\[
		T_\eps^\ast = \sup\{t \leq T_0 \colon Y_\eps(t) \leq Y_0\}.
	\]
	Then, if $T_\eps^\ast < T_0$, we know by continuity that 
	\[
		Y_\eps(T_\eps^\ast) = Y_0 \leq \exp ( C_\ast T_\eps^\ast(1 +   \tilde{\delta}_\eps Y_0^3 ))
	\]
	Hence, 
	\[
		T_\eps^\ast \geq \frac{\log Y_0}{C_\ast(1 +   \tilde{\delta}_\eps Y_0^3 )}.
	\]
	Therefore, the assertion follows by choosing 
	\[
		C_2 = 2 C_\ast,
	\] 
	and $\eps_1 < \eps_0$ such that
	\[
		 \tilde{\delta}_\eps \leq e^{- 6 C_\ast T_0}. \qedhere
	\]
\end{proof}

\begin{proof}[Proof of Theorem \ref{th:main}]
	Let $T_0 > 0$. By Theorem \ref{th:convergenceMacro}, it suffices to prove that
	there exists $\eps_1 > 0$ and $C_1 < \infty$ such that 
	\begin{equation}
		\label{eq:uniformParticlesSeparatedAtT}
		Y_\eps(t) \leq C_1, 
		\qquad \text{for all} \quad \eps  < \eps_1, \quad \text{and} \quad t \leq T_0.
	\end{equation}

	We argue by contradiction. Define $T_0$ to be the infimum over all times for which
	there is no pair $(\eps_1,C_1)$ such that \eqref{eq:uniformParticlesSeparatedAtT} holds,
	and assume $T_0 < \infty$. By Proposition \ref{pro:particlesRemainSeparated}, we know $T_0 > 0$.
	
	Let $0 < \theta < T_0$. Then, at time $T_\ast := T_0 - \theta$, 
	application of Lemma \ref{lem:aPosterioriEstimateY} yields
		\begin{equation}
		Y_\eps(t) \leq e^{ C_2 T_0}, 
		\qquad \text{for all} \quad \eps  < \eps_0, \quad \text{and} \quad t \leq T_\ast,
	\end{equation}
	for some $\eps_0 > 0$.
	Now, we can apply again Proposition \ref{pro:particlesRemainSeparated}, which yields 
		\begin{equation}
		Y_\eps(t) \leq 2 e^{ C_2 T_0}, 
		\qquad \text{for all} \quad \eps  < \eps_1, \quad \text{and} \quad t \leq T_\ast + \theta_1,
	\end{equation}
	for $\eps _1 > 0$ and some $\theta_1$ which depends only on $e^{ C_2 T_0}$.
	Thus, choosing $\theta < \theta_1$, we get a contradiction to the definition of $T_0$. 
\end{proof}

\section{Well-posedness of the macroscopic equations}
\label{sec:transportStokes}

In this section, we prove well-posedness of the macroscopic equation \eqref{eq:transportStokes}, which we write here as
\begin{equation}
\label{eq:TransportStokes}
\begin{aligned}
	\partial_t \rho + (u + v_0) \cdot \nabla \rho &= 0, \\
	\rho (0,\cdot) &= \rho_0, \\
	-\Delta u + \nabla p &= \rho e, \\
	\dv u &= 0,
\end{aligned}
\end{equation}
where $v_0 \in \IR^3$ and $e \in \IR^3$ are some given constants.

We are interested in classical solutions of this problem. More precisely,
for a given initial datum $\rho_0 \in X_\beta$ with $\nabla \rho_0 \in X_\beta$, we look for a
classical solution $ (\rho,u) \in W^{1,\infty}(0,T;X_\beta) \times L^\infty(0,T,W^{1,\infty})$ with 
$\nabla \rho \in L^{\infty}(0,T;X_\beta)$ for any positive time $T$. 
Here, $X_\beta$ is the space from Definition \ref{def:X}.

\begin{proposition}
	\label{pro:effectConstant}
	Let $u_0 \in \IR^3$ and assume
	 $ (\rho,u) \in L^{\infty}(0,T;X_\beta) \times L^\infty(0,T,W^{1,\infty})$ is a solution
	to problem \eqref{eq:TransportStokes} with $v_0 = u_0$. Let
	\[
		\sigma(t,x) := \rho(t,x - t u_0)
	\]
	and
	\[
		v(t,x) := u(t,x - t u_0).
	\]
	Then $(\sigma,v)$ solves \eqref{eq:TransportStokes} with $v_0 = 0$.
\end{proposition}

\begin{proof}
	The solution operator of the Stokes equations $S$ is a convolution 
	operator, and convolution commutes with translation. Therefore,
	$(S (\rho(t,\cdot) e )(x-t u_0) = (S(\sigma(t,\cdot) e)(x)$.
\end{proof}

\begin{theorem}
	\label{th:existenceTransportStokes}
	Assume  $\rho_0 \in X_\beta$ with $\nabla \rho_0 \in X_\beta$ for some $\beta > 2$. Then, Problem \eqref{eq:TransportStokes} 
	admits a unique solution $ \rho \in W^{1,\infty}(0,T;X_\beta)$ for all $T > 0$.
	Moreover, $ \nabla \rho \in L^\infty(0,T;X_\beta)$.
\end{theorem}

\begin{proof}

	By Proposition \ref{pro:effectConstant}, we only have to consider the case $v_0 = 0$.

	We prove the statement using the Banach fixed point theorem. 
	We can write problem \eqref{eq:TransportStokes} in a more compressed way as
	\begin{equation}
		\label{eq:TransportStokesCompressed}
	\begin{aligned}
	\partial_t \rho + S(\rho e) \cdot \nabla \rho &= 0, \\
	\rho (0,\cdot) &= \rho_0.
	\end{aligned}
	\end{equation}
	The strategy of the proof is the following. In the first part, we derive estimates for the linear equation 
	\begin{equation}
	\label{eq:TransportStokesLinear}
	\begin{aligned}
	\partial_t \rho + S(\tau e) \cdot \nabla \rho &= 0, \\
	\rho (0,\cdot) &= \rho_0,
	\end{aligned}
	\end{equation}
	In the second part, we show that the solution operator for this equation is a contraction on a suitable metric space
	for small times. 	
	In order to get a global in time solution, we finally derive estimates for this solution that show that no blow-up
	in finite time is possible.
	
	\step{Step 1. Estimates for the linear equation.}
We recall from Lemma \ref{lem:SContinuousXToLipschitz}
	\begin{align}
		\| S(\tau e) \|_{W^{1,\infty}} &\leq C \| \tau \|_{X_\beta}, \label{supEstiamteStokes1},
	\end{align}
	where $C$ depends only on $e$.

	We claim that the solution operator $A$ for Problem \eqref{eq:TransportStokesLinear}
	 maps $\tau \in L^\infty(0,T;X_\beta)$ to a function $\rho \in  L^\infty(0,T;X_\beta)$.
	To this end, we denote $v := S(\tau e)$. 
	Then, the solution to the transport equation
	 \eqref{eq:TransportStokesLinear} is given by
	\[
		\rho(t,x) = \rho_0 (\varphi(t,0,x)),
	\]
	where $\phi(t,\cdot, \cdot)$ is the flow of $v$ starting at time $t$. More precisely, $\varphi$ is the solution to
	\begin{equation}
		\label{eq:flow}
	\begin{aligned}
		\partial_s \varphi(t,s,x) &= v(s,\varphi(t,s,x)) \\
		\varphi(t,t,x) &= x.
	\end{aligned}
	\end{equation}
We observe that,
\begin{equation}
	\label{eq:finitePropagation}
	|\varphi(t,0,x) - x| \leq \int_0^t |v(s,\varphi(t,0,x)|\dd t 
	\leq C T \| \tau \|_{L^\infty(0,T;X_\beta)}.
\end{equation}
Thus,
\begin{equation}
	\label{eq:1+x^beta}
	(1+|x|^\beta) \leq  C_1\left(1+ T^\beta\| \tau \|_{L^\infty(0,T;X_\beta)}^\beta\right) {(1+|\varphi(t,0,x)|^\beta)}
\end{equation}
where we denote the generic constant by $C_1$ for future reference.
In particular,
\begin{equation}
	\label{eq:growthOfRho}
	 \|\rho\|_{L^\infty(0,T;X_\beta)}   \leq C_1(1+ T^\beta\| \tau \|_{L^\infty(0,T;X_\beta)}^\beta)  \|\rho_0\|_{X_\beta}.
\end{equation}

	\step{Step 2. Contraction for small times.}
	We want to prove that $A$ is a contraction in
	\begin{equation}
			Y := \overline{B_{2 C_1 \| \rho_0 \|_{X_\beta}}(0)} \subset {L^\infty(0,T;X_\beta)}
	\label{eq:definitionY}
	\end{equation}
	 for sufficiently small times $T$, where $C_1$ is the constant from \eqref{eq:1+x^beta}.
	 Choosing $T \leq (2 C_1 \| \rho_0 \|_{X_\beta})^{-1}$, we have seen in \eqref{eq:growthOfRho} that 
	 the solution operator $A$ for Problem \eqref{eq:TransportStokesLinear} maps $Y$ to itself.

	Let $\tau_1, \tau_2 \in Y$, and for $i= 1,2$, define $v_i = S(\tau_i e)$ the solutions to the Stokes equations, 
	$\varphi_i$ the corresponding flows as in \eqref{eq:flow}, and $\rho_i = A \tau_i$ the solutions to the linear transport equation \eqref{eq:TransportStokesLinear}.
	Then, for $t \leq T \leq (2 C_1 \| \rho_0 \|_{X_\beta})^{-1}$ we estimate using \eqref{eq:1+x^beta}
	and writing $L := 2 C_1 \|\rho_0\|_{X_\beta} \|\nabla \rho_0\|_{X_\beta}$
	\begin{align}
		(1+|x|^\beta) |\rho_1(t,x) - \rho_2(t,x)| 
		&= (1+|x|^\beta)|\rho_0(\varphi_1(t,0,x)) - \rho_0(\varphi_2(t,0,x))| \\
		&\leq L  |\varphi_1(t,0,x) -\varphi_2(t,0,x)| \\
		& \leq L	 \int_0^t |v_1(s,\varphi_1(t,s,x)) - v_2(s,\varphi_2(t,s,x))| \dd s \\
		& \leq L 	 \int_0^t |v_1(s,\varphi_1(t,s,x)) - v_1(s,\varphi_2(t,s,x))| \dd s \\
			& {} + L	 \int_0^t | v_1(s,\varphi_2(t,s,x)) - v_2(s,\varphi_2(t,s,x))| \dd s \\
		& \leq L  \|\nabla v_1\|_{L^\infty((0,t)\times \IR^3)} 
		\int_0^t |\varphi_1(t,s,x)) - \varphi_2(t,s,x)| \dd s \\
		{} & + L \| v_1 - v_2\|_{L^\infty((0,t)\times \IR^3)} t.
	\end{align}
	Using again Gronwall, we deduce
	\[
		\|\rho_1(t,x) - \rho_2(t,x)\|{L^\infty(0,T;X_\beta)} \leq  L T  \| v_1 - v_2\|_{L^\infty((0,T)\times \IR^3)} e^{ L  \|\nabla v_1\|_{L^\infty((0,T)\times \IR^3)} T}.
	\]
	Hence, using $\| \tau_1 \|_{L^\infty(0,T;X_\beta)} \leq 2 C_1 \|\rho_0\|_{X_\beta}$ from 
	 \eqref{eq:definitionY} together with the estimates for the Stokes equation \eqref{supEstiamteStokes1}, we conclude for all $t \leq T \leq (2 C_1 \| \rho_0 \|_{X_\beta})^{-1} $
	\[
		\|\rho_1- \rho_2\|_{L^\infty(0,T;X_\beta)} 
		\leq C L T \| \tau_1 - \tau_2 \|_{L^\infty(0,T;X_\beta)}  e^{ 2 C C_1 L \|\rho_0\|_{X_\beta} T}.
	\]
	This proves that $A$ is indeed a contraction if we choose $T$ sufficiently small.
	Therefore, the Banach fixed point theorem provides a unique solution $\rho$ up to this time $T$.	
	
	\step{Step 3. Global solution.}
	In order to get a global solution in time, we need to show that $\rho (t,\cdot)$ and
	 $\nabla \rho (t,\cdot)$ do not blow up in finite time, if $\rho$ is the solution 
	 to \eqref{eq:TransportStokes}.
	Define $v = S(\rho e)$ and $\varphi$ the flow of $v$ as before.
	We observe that 
	\[
		\|\rho(t,\cdot)\|_{L^1(\IR^3)} \leq C \|\rho(t,\cdot)\|_{X_\beta}
	\]
	and
	\[
		\|\rho(t,\cdot)\|_{L^\infty(\IR^3)} \leq  \|\rho(t,\cdot)\|_{X_\beta}
	\]	
	Clearly, the spatial $L^\infty$-norm of $\rho$ is conserved over time.
	Since $v$ is divergence free, also the spatial $L^1$-norm is conserved.
	Using the explicit convolution formula for the solution operator $S$ yields
	\begin{equation}
		\|v(t,\cdot)\|_{W^{1,\infty}(\IR^3)} 
		\leq C(\|\rho(t,\cdot)\|_{L^1(\IR^3)} + \|\rho(t,\cdot)\|_{L^\infty(\IR^3)})
		\leq C \|\rho_0\|_{X_\beta}.
	\label{eq:vLipschitzConstant}	
	\end{equation}
Therefore, we estimate analogously as we have obtained \eqref{eq:1+x^beta}
	\begin{align}
		(1+|x|^\beta)|\rho(t,x)| = (1+|x|^\beta)|\rho_0(\varphi(t,0,x))| \leq 
		C\left(1 + t^\beta \|\rho_0\|^\beta_{X_\beta}\right)\|\rho_0\|_{X_\beta},
	\end{align}
	and we conclude
	\[
		\|\rho\|_{L^\infty(0,T;X_\beta)} \leq C\left(1 + T^\beta \|\rho_0\|^\beta_{X_\beta}\right)\|\rho_0\|_{X_\beta}.
	\]

	In order to get estimates for the gradient of $\rho$, we differentiate equation \eqref{eq:TransportStokesCompressed} and obtain
	\[
		\partial_t \partial_{x_i} \rho = v \nabla \cdot \partial_{x_i} \rho + \partial_{x_i} v \cdot \rho. 
	\]
	Hence,
	\[
		\partial_{x_i} \rho(t,x) = \partial_{x_i} \rho_0(\varphi(t,0,x)) 
		+ \int_0^t \partial_{x_i} v(s,\varphi(t,s,x)) \cdot \rho(s,\varphi(t,s,x)) \dd s.
	\]
	Using \eqref{eq:vLipschitzConstant} leads to
	\begin{align}
		\|\nabla \rho \|_{L^\infty(0,T;X_\beta)} 
		\leq C\left(1 + T^\beta \|\rho_0\|^\beta_{X_\beta}\right)
		\left(\|\nabla \rho_0\|_{X_\beta} + T \|\rho_0\|_{X_\beta}  \|\rho\|_{L^\infty(0,T;X_\beta)} \right)
	\end{align}
	
	Therefore, both $\rho$ and $\nabla \rho$  do not blow up in finite time.
	Thus, by a standard contradiction argument using Step 2, solutions
	to \eqref{eq:TransportStokes} exist and are unique for arbitrary times $T$.
\end{proof}

\section*{Acknowledgement}

The author acknowledges support through the CRC 1060, the mathematics of emergent effects, of the University of Bonn,
that is funded through the German Science Foundation (DFG).

\printbibliography

\end{document}